\documentclass[11pt,oneside]{amsart}
\usepackage{amssymb, amsmath, amsthm}

\usepackage{graphicx}
\usepackage{color}
\usepackage{enumerate}
\usepackage{hyperref}
\usepackage{amsrefs}
\usepackage{fullpage}
\usepackage{pinlabel}

\theoremstyle{plain}
\newtheorem{theorem}{Theorem}[section]

\newtheorem*{theorem*}{Theorem}

\newtheorem{proposition}[theorem]{Proposition}
\newtheorem{corollary}[theorem]{Corollary}
\newtheorem{lemma}[theorem]{Lemma}

\theoremstyle{definition}
\newtheorem{remark}[theorem]{Remark}
\newtheorem{question}[theorem]{Question}
\newtheorem{definition}[theorem]{Definition}

\newtheorem*{conjecture}{Conjecture}

\theoremstyle{definition}

\newcommand{\Z}{\mathbb Z}

\newcommand{\nil}{\varnothing}

\newcommand{\wihat}{\widehat}
\newcommand{\defn}[1]{\emph{#1}}
\newcommand{\boundary}{\partial}
\newcommand{\mc}[1]{\mathcal{#1}}

\newcommand{\nbhd}{\operatorname{Nb}} %neigbhorood
\newcommand{\vol}{\operatorname{vol}} %neigbhorood

\renewcommand{\b}{\frak{b}}

\newcommand{\co}{\mskip0.5mu\colon\thinspace}

\begin{document}

   % title

   \title[]{Hyperbolic Brunnian Theta Curves}
   \author{Luis Celso Chan Palomo, Scott A. Taylor}
   % Note that the short title for running heads goes in square
   % brackets.  This is optional.  The long title goes in curly
   % braces.  In the long title, line breaks are indicated by \\.

  \begin{abstract}
A nontrivial $\theta$-curve in $S^3$ is Brunnian if each of its cycles is the unknot. We show that if the exterior of a Brunnian $\theta$-curve is atoroidal, then it does not contain an essential annulus. Previously, Ozawa-Tsutsumi showed that there is no essential disc. Consequently, by Thurston's work, the exterior of an atoroidal Brunnian $\theta$-curve is hyperbolic with totally geodesic boundary. It follows that Brunnian $\theta$-curves of low bridge number have exteriors that are hyperbolic with totally geodesic boundary. We also show that two Brunnian $\theta$-curves are isotopic if and only if they are neighborhood isotopic and classify Brunnian spines of genus 2 handlebody knots. We rely heavily on a classification of annuli in the exteriors of genus two handlebody knots by Koda-Ozawa and further developed by Wang in conjunction with sutured manifold theory results of Taylor.
    \end{abstract}

   % today's date, or fill in whatever date you prefer
   \date{\today}
\thanks{}
% This ends the top matter information.
% We can now tell LaTeX to display the top matter.

   \maketitle
% The body  

\section{Introduction}
Theta curves are spatial graphs that seem to lie between knots and two-component links in their complexity. They show up in a number of places in knot theory, such as in the theory of tunnel number one knots \cite{ChoMcCullough}, strongly invertible knots \cite{Sakuma}, knotoids \cite{Turaev}, and in studying knotted DNA and knotted proteins \cites{DNATheta,ProteinTheta}. Brunnian theta curves form an interesting and important subclass. For example, the Kinoshita graph (Figure \ref{Kinoshita}), the best known Brunnian $\theta$-curve, is often used as a test for determining whether a spatial graph invariant detects the knottedness of a spatial graph \emph{qua} spatial graph or is only detecting the knottedness of its constitutent knots \cites{Kinoshita, McAteeSilverWilliams,GilleRobert}. Also adding to their interest, Wolcott showed that the class of Brunnian $\theta$-curves is closed under the operation of vertex sum $\#_3$ (see Proposition \ref{Brunn vertex sum} below). Recently \cite{T3}, the second author found a correspondence between Brunnian $\theta$-curves and nontrivial 1-tangles admitting two distinct unknot closures. Some constructions of Brunnian $\theta$-curves, including examples with an essential torus in their exterior, were given in \cite{Involve}. See \cite{Taylor-review} for a survey of topics related to Brunnian theta-curves and other abstractly planar spatial graphs. It also includes constructions of Brunnian $\theta$-curves. 

\subsection{Hyperbolicity of Brunnian $\theta$-curves}
The main purpose of this paper is to show:

\begin{theorem}\label{Main theorem}
    The exterior of an atoroidal Brunnian $\theta$-curve is anannular.
\end{theorem}
  
\begin{figure}[ht!]
\centering
\includegraphics[scale=0.75]{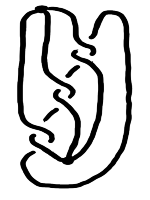}
\caption{The Kinoshita graph is the best known example of a Brunnian $\theta$-curve.}
\label{Kinoshita}
\end{figure}

This theorem is of particular interest due to the relationship of essential discs, annuli, and tori to hyperbolicity. There are two notions of hyperbolicity for spatial graphs: hyperbolic-with-parabolic meridians and hyperbolic with geodesic boundary. Let $X(\Gamma)$ be the exterior of a trivalent spatial graph $\Gamma$. Following \cite{HHMP}, we say that $\Gamma$ is \defn{hyperbolic with parabolic meridians} (henceforth, \defn{pm-hyperbolic}) if $X(\Gamma)\setminus \gamma$ admits a complete hyperbolic metric with totally geodesic boundary, where $\gamma$ is the union of the meridians of edges of $\Gamma$ with any knot components of $\Gamma$. The graph $\Gamma$ is \defn{hyperbolic} (henceforth, \defn{tg-hyperbolic}) if $X(\Gamma)$ admits a complete hyperbolic with totally geodesic boundary. 

As the authors of \cite{HHMP} remark, it is much more difficult for a spatial graph to be tg-hyperbolic than to be pm-hyperbolic. Being tg-hyperbolic is most naturally thought of as being a property of the handlebody-knot (or link) that is a regular neighborhood of the spatial graph; choosing a different spine for that handlebody does not change its exterior. On the other hand, being pm-hyperbolic is a property of the spatial graph \emph{qua} spatial graph, since the meridians of the edges are recorded. Both notions, unlike the situation with knot connected sum, are invariant under vertex summing (Proposition \ref{hyperbolicity is additive}).

Work of Thurston (see Corollary 2.5 and Theorem 2.6 of \cite{HHMP}) implies that a nontrivial spatial graph $\Gamma$ in $S^3$ is pm-hyperbolic if and only if its exterior does not admit an essential disc or annulus disjoint from the meridians of edges of $\Gamma$, as well as not admitting an essential sphere or torus; it is tg-hyperbolic if and only if its exterior does not admit an essential sphere, disc, annulus, or torus. He also showed (see \cite{Thurston-book}) that the Kinoshita graph is tg-hyperbolic.  There are examples of other Brunnian $\theta$-curves having essential tori in their exterior \cite{Involve}; those are neither pm-hyperbolic nor tg-hyperbolic. As a result of Thurston's work, Theorem \ref{Main theorem} immediately implies the following:

\begin{corollary}\label{hyperbolic thm}
    If $\Gamma$ is a Brunnian $\theta$-curve whose exterior is atoroidal then the exterior of $\Gamma$ is both pm-hyperbolic and tg-hyperbolic.
\end{corollary}

In combination with recent work of Taylor-Tomova, this has the following pleasing corollary, an analogue of the fact that most 2-bridge knots are hyperbolic:
\begin{corollary}\label{7/2 bridge}
    Suppose that $T \subset S^3$ is a Brunnian $\theta$-curve of bridge number at most 7/2. Then $T$ is both pm-hyperbolic and tg-hyperbolic.
\end{corollary}

\begin{remark}
    The \defn{bridge number} $\b(\Gamma)$ of a spatial graph $\Gamma$ in $S^3$ is the minimum of $|H \cap \Gamma|/2$, where $H \subset S^3$ is a 2-sphere transverse to $\Gamma$ dividing $S^3$ into two 3-balls $B_1, B_2$, such that for each $i \in \{1,2\}$, $\Gamma \cap B_i$ is a $\boundary$-parallel (in $B_i$) acyclic graph. See \cite{TT-g2} for more detail. The constructions of Brunnian $\theta$-curves in \cite{Involve} have bridge number at most 3. The Kinoshita graph has bridge number 5/2.
\end{remark}

\begin{proof}
By Corollary \ref{hyperbolic thm} it suffices to show that a $\theta$-curve $\Gamma$ admitting an essential torus in its exterior has bridge number at least 4. In spirit, this follows from Schubert's theorem concerning the bridge number of satellite knots \cites{Schubert, Schultens}. However, there are some subtleties for spatial graphs which were worked out by Taylor and Tomova \cite{TT-Satellite}.

Suppose that $\Gamma \subset S^3$ is a $\theta$-curve admitting an essential torus $Q$ in its exterior. Then $Q$ bounds a solid torus $V \subset S^3$ containing $\Gamma$. The \defn{wrapping number} $\omega$ of $\Gamma$ in $Q$ is the minimum of $|\Gamma \cap D|$ over all essential discs $D \subset V$ transverse to $\Gamma$. If $\omega = 1$, then we may compress $Q$ along a disc intersecting $\Gamma$ in a single point to obtain a sphere $P$ intersecting $\Gamma$ transversally in two points. This implies that $\Gamma$ is the connected sum of a (possibly trivial) $\theta$-curve with a nontrivial knot $K'$. But this, in turn, would imply that some constituent knot of $\Gamma$ is nontrivial, a contradiction. Thus, $\omega \geq 2$.

Let $K$ be the knot that is a core loop of $V$. Corollary 6.13 of \cite{TT-Satellite} says that the bridge number of $\Gamma$ is at least $\omega \b(K)$ where $\b(K)$ is the bridge number of $K$. The bridge number of a nontrivial knot is at least $2$. Thus,
\[
7/2 \geq \b(\Gamma) \geq \omega \b(K) \geq 2\b(K) \geq 4.
\]
This is a contradiction.
\end{proof}

\begin{remark}
    It is likely that the techniques of \cite{Schultens} or \cite{TT-Satellite} can be used to improve Corollary \ref{7/2 bridge} to say that a Brunnian $\theta$-curve of bridge number at most 4 is atoroidal and, therefore, pm-hyperbolic and tg-hyperbolic: If $\Gamma$ has bridge number 4, then both vertices lie on the same side of a minimal bridge sphere $H$. Since, by the previous computation, $K$ must be be 2-bridge, it is likely that $Q$ can be isotoped to intersect $H$ in four simple closed curves bounding disjoint discs in $H$ and which cut $Q$ into annuli. This is, however, not possible since on one side of $H$, $\Gamma$ has at most three components.
\end{remark}

\subsection{Spines of Brunnian genus 2 handlebody knots}

If $\Gamma \subset S^3$ is a connected spatial graph, then a closed regular neighborhood $V=\nbhd(\Gamma)$ is a \defn{handlebody knot} with \defn{spine} $\Gamma$. Equivalently, $V$ is a handlebody embedded in $S^3$ that deformation retracts to $\Gamma$. Two connected spatial graphs $\Gamma_1$ and $\Gamma_2$ are \defn{equivalent} if there is an ambient isotopy taking one to the other; they are \defn{neighborhood equivalent} if there is an ambient isotopy taking a regular neighborhood of one to a regular neighborhood of the other; that is, if the corresponding handlebody knots are equivalent. Any handlebody knot of genus at least 2 admits infinitely many inequivalent spines. In particular, although equivalence of two spatial graphs implies neighborhood equivalence, neighborhood equivalence does not imply equivalence. For instance, if $K$ is a tunnel number 1 knot with tunnel $\tau$, then $K \cup \tau$ is neighborhood equivalent to the trivial $\theta$-curve.

However, we prove:

\begin{theorem}\label{unique spine}
Two Brunnian $\theta$-curves are equivalent if and only if they are neighborhood equivalent. In particular, a genus 2 handlebody knot admits at most one Brunnian spine, up to equivalence.
\end{theorem}

 Consequently, it makes sense to speak of a \defn{Brunnian genus 2 handlebody knot}: a genus two handlebody knot admitting a Brunnian $\theta$-curve as a spine. When it does not take us too far afield, we also consider other genus 2 spatial graphs: handcuff graphs and 2-bouquets.  

Theorem \ref{unique spine} is proved using some of the same tools as in our proof of Theorem \ref{Main theorem}.

\subsection{Questions and Conjectures}
We conclude this introduction with some questions and conjectures.

\subsubsection{Essential Annuli?}

Our proof of Theorem \ref{Main theorem} raises the following question:
\begin{question}
    Does there exist a (necessarily toroidal) Brunnian $\theta$-curve admitting an essential annulus in its exterior?
\end{question}

\subsection{Unique Complements?}
Based on Theorem \ref{unique spine}, we conjecture\footnote{On the basis of preliminary versions of the work in this paper, this conjecture was stated in \cite{Taylor-review}}:

\begin{conjecture}
  Brunnian $\theta$-curves are determined by their complements. That is, if $\Gamma_1$ and $\Gamma_2$ are Brunnian $\theta$-curves in $S^3$ with homeomorphic complements, there exists a homeomorphism of $S^3$ to itself taking $\Gamma_1$ to $\Gamma_2$.
\end{conjecture}

Here are some additional reasons to believe the conjecture might be true:
\begin{itemize}
    \item Nontrivial knots are Brunnian links of one component. Gordon and Luecke \cite{GL} showed that knots are determined by their complements.
    \item Mangum and Stanford \cite{MangumStanford} showed that most Brunnian \emph{links} are determined by their complements. The cases not covered by their theorem are, to our knowledge, still open.
    \item Apart from edge sliding to produce different spines of the same handlebody, the other known construction for producing distinct spatial graphs with homeomorphic exteriors consists of twisting an essential planar surface in a spatial graph exterior. The intersection of a disc spanning a constituent unknot with the exterior of a Brunnian $\theta$-curves is an essential planar surfaces. However, twisting along such a disc has the effect of twisting the other constituent knots. In almost all cases, twisting an unknot along a disc produces a nontrivial knot \cite[Theorem 4.2]{KMS}.
\end{itemize}

\subsection{Volumes?}
For a trivalent hyperbolic graph $\Gamma$ having at least one vertex properly embedded in a closed 3-manifold $M$, we can consider its  volume $\vol_{tg}(M,\Gamma)$ of $X(\Gamma)$ using a complete hyperbolic metric with totally geodesic boundary and cusps corresponding to knot components of $\Gamma$ \cite{HHMP}. Similarly, if $\Gamma$ is hyperbolic with parabolic meridians, we may consider the volume $\vol_{pm}(M,\Gamma)$ of the noncompact hyperbolic manifold with totally geodesic boundary consisting of pairs-of-pants and annular cusps corresponding to the meridians of $T$ and toroidal cusps corresponding to the knot components of $T$. In Proposition \ref{hyperbolicity is additive} below, we prove that the vertex sum of two Brunnian $\theta$-curves is hyperbolic (in either sense) if and only if each of the factors are. Thus, we ask:
\begin{question}
Suppose that $(M_1, \Gamma_1)$ and $(M_2, \Gamma_2)$ are hyperbolic or hyperbolic with parabolic meridians $\theta$-curves. Let $(M,\Gamma) = (M_1,\Gamma_1) \#_3 (M_2, \Gamma_2)$ be any trivalent vertex sum. How do the invariants $\vol_{tg}(M,\Gamma)$ and $\vol_{tg}(M_1, \Gamma_1) + \vol_{tg}(M_2, \Gamma_2)$ relate to each other? How do the invariants $\vol_{pm}(M,\Gamma)$ and $\vol_{pm}(M_1, \Gamma_1) + \vol_{pm}(M_2, \Gamma_2)$ relate to each other?
\end{question}

In his original article on knotoids \cite{Turaev}, Turaev establishes a relationship between spherical knotoids and theta-curves with a distinguished constituent unknot. He shows that there is a semigroup isomorphism from spherical knotoids to such theta curves. Essentially, this works as follows. A \defn{spherical knotoid} is (informally) an immersed arc in the sphere with over/under information at the crossings. Considering the sphere as a Heegaard sphere for $S^3$, a $\theta$-curve may be created by coning the endpoints of the arc to the centers of each of the balls on either side of the sphere. Conversely, if $T$ is a $\theta$-curve containing a distinguished constituent unknot $U$, we may arrange that $U$ is in 1-bridge position with respect to a Heegaard sphere and may then generically project the third edge onto the Heegaard sphere to obtain a knotoid. 

Recently the authors of \cite{Adams1} and \cite{Adams2} defined two notions of hyperbolicity for spherical knotoids, each giving rise to a definition of hyperbolic volume. These  hyperbolic volumes are additive under knotoid product. Neither of these definitions coincides with the corresponding $\theta$-curve being pm-hyperbolic or tg-hyperbolic; these therefore give two other possibilities for defining the hyperbolicity of spherical knotoids. It would be interesting to pin down the relationship between these different versions.

\section{Definitions and Conventions}\label{sec def}

We denote the exterior of a set $U \subset S^3$ by $X(U)$ and a regular neighborhood (open or closed, depending on context) by $\nbhd(U)$. 

A \defn{spatial graph} (respectively \defn{handlebody knot}) is a graph (resp. handlebody) properly embedded in a 3-manifold; in this paper that 3-manifold will be $S^3$. A graph is \defn{trivalent} if every vertex is trivalent. A spatial graph $\Gamma$ is \defn{reducible} if there exists a properly embedded sphere in $M$ intersecting $\Gamma$ transversally in a single point.  A \defn{$\theta$-curve} is a connected spatial graph with two vertices, three edges, and no loops.  A \defn{handcuff curve} is a connected spatial graph with three edges, two of which are loops based at distinct vertices; it has a single constituent two-component link formed by the loops. A \defn{2-bouquet} is a connected spatial graph with a single vertex and two loops; it has two constituent knots meeting in a single point. A spatial graph is a \defn{genus 2 curve} if it is a $\theta$-curve, handcuff curve, or 2-bouquet. For genus 2 curves in $S^3$, only a handcuff curve can be reducible. A \defn{constituent knot} (resp. link) in a spatial graph is a cycle (resp. union of disjoint cycles) of the spatial graph. If $\Gamma$ is a $\theta$-curve with edges $e_1, e_2, e_3$, we let $K_{ij} = e_i \cup e_j$.

A spatial graph in $S^3$ is \defn{trivial} or \defn{unknotted} if it can be isotoped to lie on a tame 2-sphere. A spatial graph in $S^3$ has the \defn{Brunnian property} if each proper subgraph is trivial. In particular, a $\theta$-curve or 2-bouquet has the Brunnian property if and only if every constituent knot is unknotted. A handcuff curve has the Brunnian property if and only if its constituent two-component link is the unlink.  A spatial graph is \defn{Brunnian} if it is nontrivial and has the Brunnian property. In the literature, a Brunnian $\theta$-graph is also known as \emph{almost unknotted} or \emph{minimally knotted} or a \defn{ravel}.  For other spatial graphs, these terms mave have slightly different meanings, depending on the author. Notice that if $\Gamma$ is a reducible handcuff curve with the Brunnian property then it must be the trivial handcuff curve. 

Given spatial graphs $\Gamma_1$ and $\Gamma_2$ (possibly in different 3-manifolds), having trivalent vertices $v_1$ and $v_2$, we can form a trivalent vertex sum $\Gamma_1 \#_3 \Gamma_2$ as follows. Remove a regular neighborhood of each of $v_1$ and $v_2$ from the ambient 3-manifolds and glue the resulting spherical boundary components $S_1, S_2$ together by a homeomorphism $h \co S_1 \to S_2$ taking $\Gamma_1 \cap S_1$ to $\Gamma_2 \cap S_2$. The result is a spatial graph $\Gamma_1 \#_3 \Gamma_2$ in the connected sum of the 3-manifolds.  This sum is particularly well-behaved on $\theta$-curves in $S^3$. If $\Gamma_1$ and $\Gamma_2$ are \emph{oriented} $\theta$-curves and if $h$ preserves the orientations, then the resulting $\theta$-curve is unique up to equivalence. See \cite{Wolcott} for details. Matveev and Turaev \cite{MT} prove that $\theta$-curves have unique (in a certain sense) prime decompositions with respect to trivalent vertex sum. The trivial $\theta$-curve is an identity element for $\#_3$. There is a \defn{summing sphere} $S$, intersecting $\Gamma_1 \#_3 \Gamma_2$ three times and dividing $\Gamma_1 \#_3 \Gamma_2$ back into $\Gamma_1$ and $\Gamma_2$. 

A \defn{meridian} of a handlebody $V$ is an essential simple closed curve in $\boundary V$ bounding a disc in $V$. If $\Gamma$ is a spatial graph, a \defn{meridian} of $\Gamma$ is a meridian of $V = \nbhd(\Gamma)$ that bounds a disc in $V$ intersecting $\Gamma$ in exactly one point, and that point is interior to an edge of $\Gamma$. Note that, up to proper isotopy, the meridians of a spatial graph are in bijection with the edges, but that a handlebody knot of genus at least two has infinitely many meridians.

Let $N$ be a compact, orientable 3-manifold with nonempty boundary. If $b \subset N$ is a simple closed curve, we denote the 3-manifold obtained by attaching a 2-handle to the curve $b$, by $N[b]$. We apply this in the case when $N$ is the exterior of a genus two handlebody $V \subset S^3$. In this case, if $b$ is a separating curve, $N[b]$ is the exterior of a two-component link in $S^3$ and if $b$ is nonseparating, then $N[b]$ is the exterior of a knot in $S^3$. Given a properly embedded surface $Q \subset N$ with $\boundary Q$ intersecting $b$ minimally up to isotopy, a \defn{$b$-boundary compressing disc} for $Q$ is a disc embedded in $N$, with interior disjoint from $\boundary N \cup Q$ and with boundary the endpoint union of an arc on $Q$ and a subarc of $b$. If $\boundary Q$ is disjoint from $b$, then there is no $b$-boundary compressing disc for $Q$.

\section{Outline}
We prove Theorem \ref{Main theorem} by combining classical 3-manifold theory techniques, the sutured manifold theory results of Taylor \cites{T1,T2}, and the classification of essential annuli in genus two handlebody exteriors by Koda and Ozawa \cite{KO} and further developed by Wang \cite{Wang} and Koda-Ozawa-Wang \cite{KOW}. 

In Section \ref{basic props} we prove some basic facts about Brunnian $\theta$-curves, some of which may be of independent interest.

In Section \ref{sec class} we summarize the Koda-Ozawa annulus classification. The Koda-Ozawa classification guarantees that (in our context) there is a certain meridional disc for the handlebody knot $V = \nbhd(\Gamma)$. (This disc may or may not be a meridian of $\Gamma$.) We find it helpful to give this disc a name.

\begin{definition}
Suppose that $V \subset S^3$ is a genus 2 handlebody knot. Suppose that $Q \subset X(V)$ is an essential annulus. A \defn{guardrail} for $Q$ is an essential disc $A \subset V$ for which $a = \boundary A$ is disjoint from $\boundary Q$. We also call the curve $a$ a \defn{guardrail} for $Q$.
\end{definition}

For our purposes, the key result from the classification is Corollary \ref{Guardrails exist}, guaranteeing the existence of guardrails.

In Section \ref{sec specific} we consider the possibility that a guardrail is actually a meridian of $\Gamma$, not just of $\nbhd(\Gamma)$. The most involved case concerns the possibility that each boundary component of an essential annulus $Q$ is a $(p,q)$ curve on the regular neighborhood of one of the cycles of $\Gamma$, while being disjoint from a neighborhood of the 3rd edge which contains the guardrail. The only time we cannot rule out the existence of such an annulus $Q$ is when the components of $\boundary Q$ are preferred longitudes of the cycle. But in that case, we can join them by an annulus lying in the boundary of the neighborhood of the cycle to produce an essential torus in the $\theta$-curve exterior.

In Section \ref{sec: SMT} we finish the proof that atoroidal Brunnian $\theta$-curve exteriors do not admit an essential annulus, by applying a result of Taylor's from sutured manifold theory to the case when the guardrail is not a meridian of $\Gamma$.  Figure \ref{fig:annularbrunnianhandcuff} shows two examples of Brunnian handcuff curves admitting essential annuli. Observe that in each example the annulus admits the meridian of the separating edge as a guardrail; Proposition \ref{limited prop} shows this will always be the case.

In Section \ref{Spine Uniqueness}, we apply that same result to understand when a genus 2 handlebody knot admits more than one Brunnian spine and prove Theorem \ref{unique spine}.

\begin{figure}\label{fig:annularbrunnianhandcuff}
\centering
\includegraphics[scale=0.5]{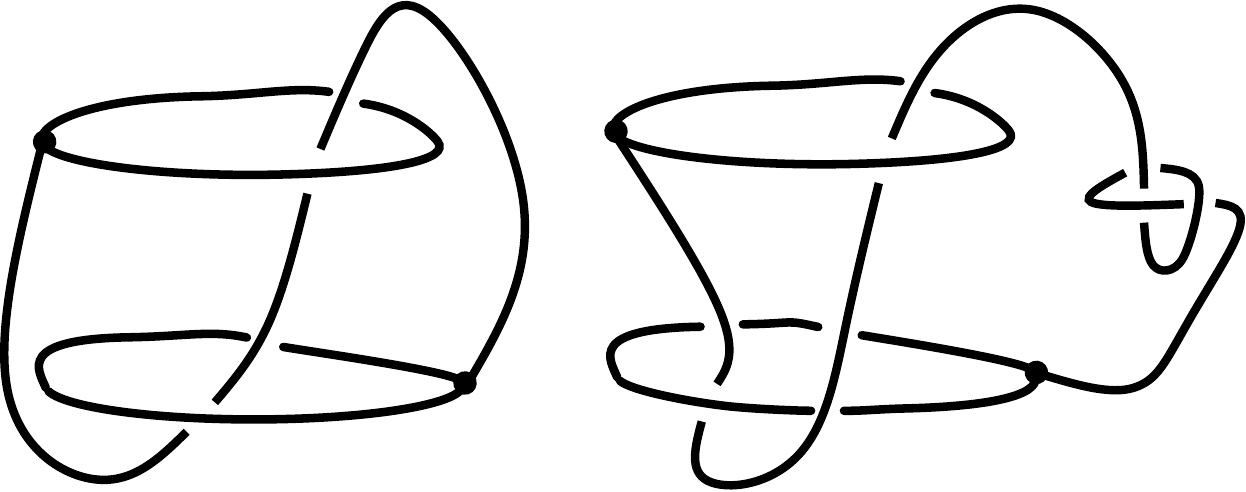}
\caption{On the left is a Brunnian handcuff curve with an essential annulus joining the two loops and disjoint from the separating edge. On the right is a Brunnian handcuff curve with an essential annulus having both boundary components meridians of the separating edge.}
\end{figure}

\section{Basic Properties}\label{basic props}

We repeatedly use the following useful result (which is a special case of a more general result):

\begin{lemma}[{Ozawa-Tsutsumi \cite{OT}}]\label{bd irred}
    If $\Gamma$ is a Brunnian $\theta$-curve or handcuff curve, then $X(\Gamma)$ is $\boundary$-irreducible 
\end{lemma}

Similarly:

\begin{lemma}\label{basic fact}
Suppose that $\Gamma$ is a Brunnian $\theta$-curve with constituent unknot $\sigma$. Then there is no essential disc in $X(\sigma)$ intersecting the edge $\Gamma\setminus \sigma$ in a single point.
\end{lemma}

In the statement of Lemma \ref{basic fact}, we mean $X(\sigma)$ to be the complement of an open regular neighborhood $\nbhd(\sigma)$ intersecting $\Gamma \setminus \sigma$  in a regular neighborhood of its endpoints.

\begin{proof}
Suppose that $\Gamma$ is a $\theta$-curve having the Brunnian property. Let $\sigma$ be a cycle whose exterior admits an essential disc $D$ intersecting $\Gamma\setminus \sigma$ in one point $p$. Let $\sigma'$ be one of the other cycles. Without loss of generality, assume that $\sigma = e_1 \cup e_2$ and $\sigma' = e_2 \cup e_3$. Let $E$ be a disc with boundary $\sigma'$. Choose it so that $E \cap \nbhd(\sigma)$ is a regular neighborhood in $E$ of the arc $e_2 \cup (e_3 \cap \nbhd(\sigma))$.

The intersection $D \cap E$ consists of $e_2$ and a collection of arcs and \emph{a priori} circles. However, there is a unique arc in $D \cap E$ with one endpoint on $p$. Consequently, an innermost disc argument there are no circles. Any other arc of intersection has both endpoints on $\boundary \nbhd(\sigma)$. An outermost arc argument shows that, by minimality, no such arcs can exist. Consequently, $D \cap E$ consists of exactly one arc $\zeta$; it has one endpoint at $p$ and one on $\boundary \nbhd(\sigma)$. In $E$ it cuts off a subdisc that can be used to isotope $\Gamma$ so as to eliminate the point $p$. Consequently, $X(\Gamma)$ is $\boundary$-reducible, contradicting Lemma \ref{bd irred}.
\end{proof}

Finally, we present two results showing that Brunnian $\theta$-curves are an interesting category. The first is due to Wolcott, and follows easily from the fact that $\#_3$ acts as a connected sum on the cycles of a $\theta$-curve. The second relies on Corollary \ref{hyperbolic thm}.

\begin{proposition}[{Lemma 4.1 and Theorem 4.2 of \cite{Wolcott}}]\label{Brunn vertex sum}
 If $\Gamma_1$ and $\Gamma_2$ are Brunnian $\theta$-curves, then so is $\Gamma_1 \#_3 \Gamma_2$. Conversely, if $\Gamma = \Gamma_1 \#_3 \Gamma_2$ is a Brunnian $\theta$-curve with $\Gamma_1, \Gamma_2$ both nontrivial, then $\Gamma_1$ and $\Gamma_2$ are both Brunnian.
\end{proposition}

\begin{proposition}\label{hyperbolicity is additive}
Suppose that $\Gamma_1$ and $\Gamma_2$ are spatial graphs in $S^3$ and that $\Gamma = \Gamma_1 \#_3 \Gamma_2$ is any trivalent vertex sum. If $\Gamma_1$ and $\Gamma_2$ are tg-hyperbolic or pm-hyperbolic, then so is $\Gamma$. On the other hand, if $\Gamma$ is a Brunnian tg-hyperbolic $\theta$-curve and $\Gamma_1$ and $\Gamma_2$ are nontrivial, then $\Gamma_1$ and $\Gamma_2$ are also tg-hyperbolic Brunnian $\theta$-curves.
\end{proposition}
\begin{proof}
We prove the contrapositive. Suppose either that $\Gamma$ is not tg-hyperbolic or that it is not pm-hyperbolic. By work of Thurston if a spatial graph is not tg-hyperbolic its exterior contains an essential sphere, disc, annulus or torus \cite[Theorem 2.6]{HHMP} and if it is not pm-hyperbolic its exterior contains an essential sphere or torus or an essential disc or annulus disjoint from meridians of its edges \cite[Theorem 2.2]{HHMP}. Let $S$ be the summing sphere corresponding to the vertex sum. 

We prove the statement for tg-hyperbolicity, as the proof for pm-hyperbolicity is nearly identical. Let $P \subset X(\Gamma)$ be an essential sphere, disc, annulus, or torus. Out of all such possibilities, choose $P$ to minimize $|S \cap P|$. If $S \cap P = \nil$, then $P$ lives in either $X( \Gamma_1)$ or $X(\Gamma_2)$. Whichever it is, cannot be tg-hyperbolic. Assume, therefore, that $S \cap P \neq \nil$.

Let $\gamma \subset S \cap P$ be a circle of intersection (if such exists), bounding an innermost disc $D \subset S$. We have $|D \cap \Gamma| \in \{0,1\}$, since $S$ is a thrice punctured sphere. Since $P$ is incompressible, if $D \cap \Gamma = \nil$, then $\gamma$ also bounds a disc $E$ in $P$ disjoint from $\Gamma$. After a small isotopy, the sphere $D \cup E$ intersects $S$ fewer times than does $P$. Thus, it is not essential and so bounds a 3-ball in $S^3$ disjoint from $\Gamma$. An isotopy of $E$ across $D$ converts $P$ into a sphere contradicting our choice of $P$. Thus, $|D \cap \Gamma| = 1$. In all cases, compressing $P$ using $D$ creates a surface contradicting our original choice of $P$. Hence, $S \cap P$ consists of arcs. In particular, $P$ is a disc or annulus.

Let $\zeta \subset S \cap P$. If $P$ is a disc, consider an outermost disc $D \subset P$ cut off by an arc of $S \cap P$. The disc $D$ is either inessential in one of $X(\Gamma_1)$ or $X(\Gamma_2)$, in which case we can isotope $P$ to reduce $|P \cap S|$ or it is essential. In which case, $\Gamma_1$ or $\Gamma_2$ is nonhyperbolic. Consider, therefore, the case when $P$ is an annulus. As before, if some arc of $S \cap P$ is inessential in $P$, we contradict our choice of $P$. Thus, all arcs are essential in $P$. A component of $P\setminus S$ is, therefore, an essential disc in $X(\Gamma_1)$ or $X(\Gamma_2)$. Thus, at least one of $\Gamma_1$ or $\Gamma_2$ is not tg-hyperbolic. 

On the other hand, suppose that $\Gamma$ is a Brunnian $\theta$-curve and that $\Gamma_1, \Gamma_2$ are nontrivial. By Proposition \ref{Brunn vertex sum}, they are Brunnian. If one (say $\Gamma_1$) is not tg-hyperbolic, then by Corollary \ref{hyperbolic thm}, $X(\Gamma_1)$ admits an essential torus. When creating $\Gamma$, as the vertex of $\Gamma_1$ where the sum is performed is disjoint from the torus, the essential torus persists to $X(\Gamma)$, showing that $X(\Gamma)$ is not tg-hyperbolic.
\end{proof}

\section{Koda-Ozawa Annuli Classification}\label{sec class}

In this section, we explain the Koda-Ozawa classification \cite{KO} of annuli in the exterior of a genus 2 handlebody knot. Along the way, we make use of certain rephrasings and improvements found in subsequent papers \cites{Wang, KOW}. We also adjust the notation to coincide with that used in the rest of this paper; however, not all the terminology will be used in the remainder.

\begin{definition}[{\cite[pp3-4]{KOW}}]
Suppose that $V \subset S^3$ is a genus two handlebody knot and that $Q \subset X(V)$ is an essential annulus. Then:
\begin{enumerate}
\item $Q$ is \defn{Type 1} if both components of $\boundary Q$ are meridians of $V$;
\item $Q$ is \defn{Type 2} if exactly one component of $\boundary Q$ is a meridian in $V$;
\item $Q$ is \defn{Type 3} if no component of $\boundary Q$ is a meridian of $V$, but there is a compressing disc $D$ for $\boundary V$ in $S^3$ disjoint from $\boundary Q$. Additionally, for some such $D$:
\begin{enumerate}
    \item $Q$ is \defn{Type 3-1} if $D \subset X(V)$
    \item $Q$ is \defn{Type 3-2} if $D \subset V$, $D$ does not separate $\boundary Q$, the components of $\boundary Q$ are parallel in $\boundary V$;
    \item $Q$ is \defn{Type 3-3} if $D \subset V$, $D$ does separate $\boundary Q$, the components of $\boundary Q$ are not parallel in $\boundary V$;
\end{enumerate}
\item $Q$ is of \defn{Eudave-Mu\~noz Type} if $V$ is the double-branched cover of a Eudave-Mu\~noz tangle and $Q$ is the standard associated annulus.
\item $Q$ is of \defn{Type 4} if the components of $\boundary Q$ are parallel in $\boundary V$ and there is no compressing disc for $\boundary V$ in $S^3$ disjoint from $Q$. Additionally,
\begin{enumerate}
    \item $Q$ is of \defn{Type 4-1} if $X(V)$ is atoroidal\footnote{Note the typo in \cite{KOW}.};
    \item $Q$ is of \defn{Type 4-2} if $X(V)$ is toroidal.
\end{enumerate}
\end{enumerate}
\end{definition}

\begin{remark}\label{EM type}
    By \cite{KO} and \cite[Lemma 4.1]{KOW}, every annulus of type 4-1 is of Eudave-Mu\~noz Type.
\end{remark}

\begin{lemma}\label{exists 123}
    If $X(V)$ is $\boundary$-irreducible and atoroidal and $X(V)$ admits an essential annulus, then $X(V)$ admits an essential annulus of Type 1, 2, 3-2, or 3-3. 
\end{lemma}
\begin{proof}
    By \cite{KO}, every essential annulus $Q$ in the exterior of a genus two handlebody knot is of Type 1, 2, 3-1, 3-2, 3-3, 4-1, or 4-2. Since $X(V)$ is atoroidal, there is no annulus of Type 4-2. By \cite{KOW}, if $V$ admits an annulus of Type 4-1, then it also admits an essential annulus of Type 3-2. As this is not explicitly stated in \cite{KOW}, we show how it follows from the results that are stated. We assume familiarity with the terminology of that paper.

    Suppose that $V$ admits an annulus of Type 4-1. By Remark \ref{EM type}, the annulus is of Eudave-Mu\~noz Type. Theorem 2.2, Lemma 3.2, and Definition 3.3, all in \cite{KOW}, imply that either the pair $(S^3, V)$ is of type $K$ or is of Type $M$. In the first case, the exterior of $V$ has exactly one characteristic annulus. In which case, by \cite[Lemma 4.9]{KOW} (which is a rewriting of \cite[Lemma 2.3]{Wang}) and the notation described on \cite[Page 3--4]{KOW}, that annulus is an essential annulus of Type 3-2. On the other hand, if $(S^3, V)$ is of Type $M$, then $E(V)$ has two non-characteristic annuli, one of which is Type 4-1. By \cite[Theorem 5.2]{KOW}, the other non-characteristic annulus is Type 3-2.  (Also by \cite[Lemma 2.3, Theorem 4.3, Theorem 2.3]{KOW}, there are two characteristic annuli both of type 3-2.)    

    Type 3-1 cannot occur by the hypothesis that $X(V)$ is $\boundary$-irreducible. 
 \end{proof}

\begin{corollary}\label{Guardrails exist}
Suppose that $V \subset S^3$ is a $\boundary$-irreducible, atoroidal genus 2 handlebody knot admitting an essential annulus $Q$ in its exterior. Then there is such an annulus admitting a guardrail.
\end{corollary}
\begin{proof}
 By Lemma \ref{exists 123}, we may assume that $Q$ is of type 1, 2, 3-2, or 3-3. If $Q$ is of type 1 or 2, then a pushoff of a component of $\boundary Q$ is a meridian. Types 3-2 and 3-3 admit a guardrail by definition. 
\end{proof}

\begin{remark}
In what follows, we will use the specific guardrails for each type of annulus, as in the proof of Corollary \ref{Guardrails exist}. 
\end{remark}

\section{Specific annulus configurations}\label{sec specific}

In this section, we examine the cases where the guardrail for the putative annulus $Q$ is a meridian, not just of $\nbhd(\Gamma)$, but also of $\Gamma$ itself. The most involved situation is when $Q$ is a Type 3-2 annulus. We split that case into two subcases. The first is when $Q$ is what we call a ``tubed unknotting disc annulus'' (TUDA).  A TUDA is constructed by tubing together two parallel copies of an unknotting disc for one of the cycles of $\Gamma$. It is a particular type of Type 3-2 annulus having a guardrail that is a meridian of $\Gamma$.

\begin{definition}\label{TUDA}
    Suppose that $\Gamma$ is a $\theta$-curve having a constituent unknot $\sigma$ with the property that $\sigma$ does not bound a disc with interior disjoint from $\Gamma$. Suppose that $Q$ is an essential annulus properly embedded in the exterior of $\Gamma$ such that $\boundary Q$ is a pair of essential curves in $\boundary \nbhd(\sigma)$, and so that $Q$ is compressible in the exterior of $\sigma$. (Necessarily, $\boundary Q$ consists of a pair of preferred longitudes on $\boundary \nbhd(\sigma)$. We say that $Q$ is a \defn{tubed unknotting disc annulus (TUDA)} See Figure \ref{fig TUDA}.
\end{definition}

\begin{figure}
\includegraphics[scale=0.4]{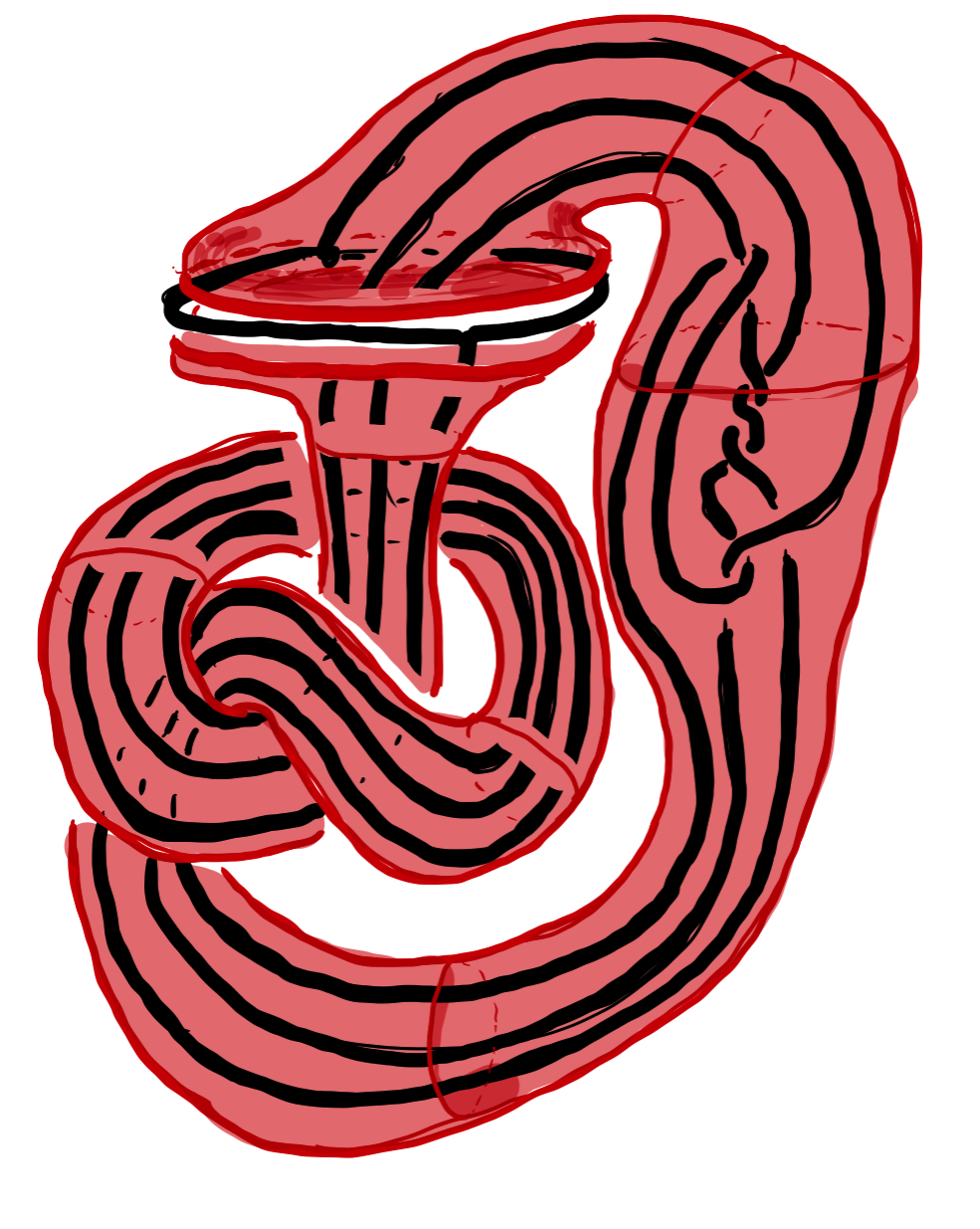}
\caption{An example of a tubed unknotting disc annulus (TUDA). Its boundary components are preferred longitudes on a cycle $\sigma \subset \Gamma$.}
\label{fig TUDA}
\end{figure}

When $Q$ is a TUDA, we show we can attach an annulus to $\boundary Q$ to create an essential torus in $X(\Gamma)$. The second subcase is when $Q$ is not a TUDA. In this case, we use techniques inspired by \cite[Section 3]{GoLi} to find a contradiction. We prove:

\begin{proposition}\label{Step 1}
If $\Gamma$ is a Brunnian $\theta$-curve such that $X(\Gamma)$ contains an essential annulus $Q$ of Type 1, 2, 3-2, or 3-3 with guardrail a meridian of $\Gamma$, then $Q$ is a TUDA and $X(\Gamma)$ contains an essential torus.
\end{proposition}

The rest of this section is devoted to the proof of Proposition \ref{Step 1}. Throughout this section, let $\Gamma$ be a $\theta$-curve and $Q \subset X(\Gamma)$ an essential annulus. In each case, we show that either $\Gamma$ cannot be Brunnian or (when $Q$ is a TUDA) that $X(\Gamma)$ admits an essential torus.  Consider the options:

\subsection{$Q$ is a Type 1 annulus with each component of $\boundary Q$ a meridian of $\Gamma$}

 We can attach meridional discs in $\nbhd(\Gamma)$ to $\boundary Q$ to get a sphere $P$ intersecting $\Gamma$ in two points. The sphere $P$ separates $S^3$ into two 3-balls. One of the 3-balls $B$ must intersect $\Gamma$ in a subarc $\tau$ of an edge of $\Gamma$. If $\tau$ were knotted, at least two constituent knots of $\Gamma$ would be nontrivial, implying $\Gamma$ is not Brunnian. On the other hand, if $\tau$ were unknotted, this would imply $Q$ was $\boundary$-parallel in $X(\Gamma)$, a contradiction.

\subsection{$Q$ is a Type 2 annulus with one boundary component a meridian of $\Gamma$}

Suppose $Q$ is a Type 1 annulus with one component of $\boundary Q$ a meridian of $\Gamma$. Attach a meridional disc in $\nbhd(\Gamma)$ to that component of $\boundary Q$ to create a disc $P$. Then $P$ is a $\boundary$-compressing disc for the exterior of a cycle $\sigma$ of $\Gamma$; it intersects the edge $\Gamma\setminus \sigma$ exactly once. By Lemma \ref{basic fact}, $\Gamma$ is not Brunnian.

\subsection{Type 3-3 annulus with guardrail a meridian of $\Gamma$}

$Q$ cannot be of Type 3-3, since each meridian of $\Gamma$ is nonseparating in $\boundary X(\Gamma)$.

It remains to consider the Type 3-2 annuli. Suppose $Q$ is a Type 3-2 annulus with guardrail a meridian of $\Gamma$. Let $\sigma \subset \Gamma$ be the cycle consisting of the two edges that are not dual to the guardrail. Then $X(\sigma)$ is a solid torus and $Q$ is either incompressible or compressible in $X(\sigma)$. If it is compressible, then upon compressing it, we create two essential discs in $X(\sigma)$; the annulus $Q$ is, therefore, a TUDA. On the other hand $Q$ may be incompressible in $X(\sigma)$.

\subsection{$Q$ is a Type 3-2 annulus that is a TUDA} 

\begin{lemma}\label{lem: TUDA}
Suppose that $\Gamma$ is a $\theta$-curve with $\boundary$-irreducible exterior admitting a TUDA, then $X(\Gamma)$ admits an essential torus.
\end{lemma}
\begin{proof}
Assume $Q$ is a TUDA; it is compressible by definition. Since compressing $Q$ results in two essential discs, $\boundary Q$ must consist of preferred longitudes of $\sigma$. The curves $\boundary Q$ divide $\boundary \nbhd(\sigma)$ into two annuli. Attaching either of them to $Q$ creates a torus. Since every torus separates $S^3$ and since $\Gamma$ is a $\theta$-curve, one of those annuli $A$ must be disjoint from $\Gamma$ and the other intersects $\Gamma\setminus \sigma$ twice. Let $\wihat{Q} = Q \cup A$. Isotope the torus $\wihat{Q}$ slightly so that it is properly embedded in $X(\Gamma)$. It contains $\Gamma$ on the side to which $Q$ compresses; there is a unique isotopy class of curve on $\wihat{Q}$ bounding a compressing disc in $S^3$ to that side. Since the core curve of $A$ is a preferred longitude for the unknot $\sigma$, that must represent the isotopy class. The torus $\wihat{Q}$ is incompressible to the side containing $\Gamma$. If it were compressible to the side not containing $\Gamma$, the torus $\wihat{Q}$ would be a Heegaard torus for $S^3$ and so the compressing disc could be isotoped so its boundary intersects $A \subset \wihat{Q}$ in a single spanning arc. This would imply that $Q$ was $\boundary$-compressible. Performing the $\boundary$-compression creates an essential disc $D$ in $X(\Gamma)$, contradicting the hypothesis that $X(\Gamma)$ is $\boundary$-irreducible.
 \end{proof}

\begin{remark}
    We do not know if it is possible for a (necessarily toroidal) Brunnian $\theta$-curve to admit a tubed unknotting disc annulus. 
\end{remark}

\subsection{$Q$ is a Type 3-2 annulus with guardrail a meridian of $\Gamma$ and $Q$ is incompressible in the exterior of the cycle $\sigma$ disjoint from the edge dual to the guardrail.} \label{nonTUDA}

Let $Q$ be such an annulus; it is not a TUDA. See Figure \ref{SFSNotation1} for a depiction of the notation used in this case. Number the edges of $\Gamma$ so that $e_3$ is dual to the guardrail that is a meridian of $\Gamma$. Let $\sigma = e_1 \cup e_2$. Let $U = \nbhd(\sigma)$. Then $\boundary Q$ consists of two parallel curves on $\boundary U$ dividing it into two annuli $A_1$ and $A_2$. Since $\sigma$ is an unknot, $Q$ is parallel to one of $A_1$, $A_2$. Since $Q$ is essential in $X(\Gamma)$, it is parallel to exactly one of them. Choose the labelling so that $Q$ is parallel to $A_1$; together they cobound a solid torus $V$ containing the arc $\psi = e_3\setminus U$. Let $S_U = \boundary U$ and $S_V = \boundary V$ so that $S_U \cap S_V = A_1$. Let $\Delta$ be a disc with boundary $K_{13} = e_1 \cup e_3 = \psi \cup \phi_1$ and transverse to $(\boundary U) \cup Q = S_U \cup S_V$. Let $e_3^\pm$ be the components of $e_3 \cap U$ and let 
$\phi_1 = e_1 \cup e_3^+ \cup e_3^-$ and $\phi_2 = e_2 \cup e_3^+ \cup e_3^-$. These are arcs properly embedded in $U$.   

Suppose the annulus $A_1$ wraps $p \in \Z$ times longitudinally around $U$ and $q \geq 0$ times meridionally around $U$. In fact, $q \geq 2$. For if $q = 0$, then $X(\Gamma)$ would be $\boundary$-reducible and if $q = 1$, then $Q$ would be $\boundary$-compressible in $X(\Gamma)$. Also, $p \neq 0$, for if $p = 0$, then $Q$ would be a Type 1 annulus with both components of $\boundary Q$ meridians of $\Gamma$.

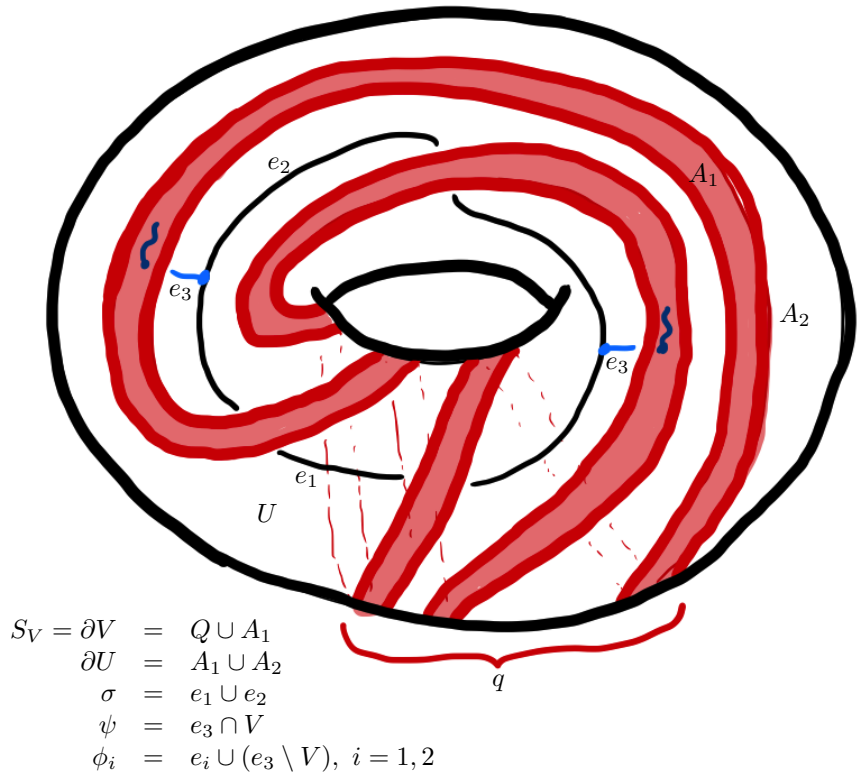
\begin{figure}
\labellist
\small\hair 2pt
\pinlabel {$q$} [t] at 217 46
\pinlabel {$A_1$} at 305 258 
\pinlabel {$U$} at 118 113
\pinlabel {$e_2$} [br] at 130 257
\pinlabel {$e_1$} [t] at 135 133
\pinlabel {$e_3$} [t] at 81 213
\pinlabel {$e_3$} [t] at 268 182
\pinlabel {$A_2$} at 344 198
\pinlabel {$\begin{array}{rcl}
S_V = \boundary V &=& Q \cup A_1\\
\boundary U &=& A_1 \cup A_2 \\
\sigma &=& e_1 \cup e_2\\
\psi &=& e_3 \cap V\\
\phi_i &=& e_i \cup (e_3 \setminus V), ~i =1,2
\end{array}
$} [l] at 0 35
% \pinlabel {$S_V = \boundary V = Q \cup A_1$} [l] at 5 70
% \pinlabel {$\boundary U = A_1 \cup A_2$} [l] at 35 53
% \pinlabel {$\sigma = e_1 \cup e_2$} [l] at 44 36
% \pinlabel {$\psi = e_3 \cap V$} [l] at 43 19
% \pinlabel {$\phi_i = e_i \cup (e_3 \setminus V), ~i =1,2$} [l] at 40 2
\endlabellist
\centering
    \includegraphics[scale=0.88]{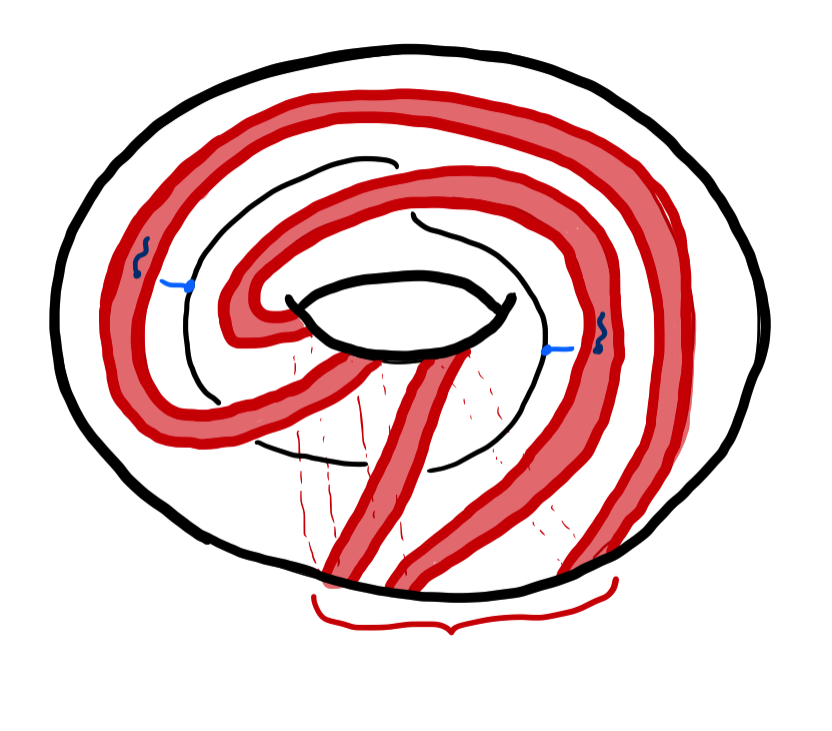}
    \caption{The notation used in Special Case \ref{nonTUDA}. We do not show the annulus $Q$ which is parallel to the annulus $A_1 \subset \boundary U$.}
    \label{SFSNotation1}
\end{figure}

\begin{definition}\label{def:normal}
We say that $\Delta$ is \defn{normal} with respect to $U$ if each component of $\Delta \cap U$ is one of the following (see Figure \ref{SFSNotation2}):
\begin{enumerate}
    \item A disc $\Delta_U$ whose boundary is the union of $\phi_1$ with an arc $\epsilon_U \subset \boundary U$ which is transverse to $\boundary A_1$. (There is a unique such component)
    \item A meridional disc $D$ of $U$ whose boundary intersects $A_1$ in exactly $p$ spanning arcs
    \item A disc $D$ such that $\boundary D$ is the frontier of a regular neighborhood of $\epsilon_U$ and, in $S_U\setminus \nbhd{\boundary \phi_1}$, intersects $\boundary A_1$ minimally up to isotopy.
\end{enumerate}
\end{definition}

\begin{figure}
    \includegraphics[scale=0.7]{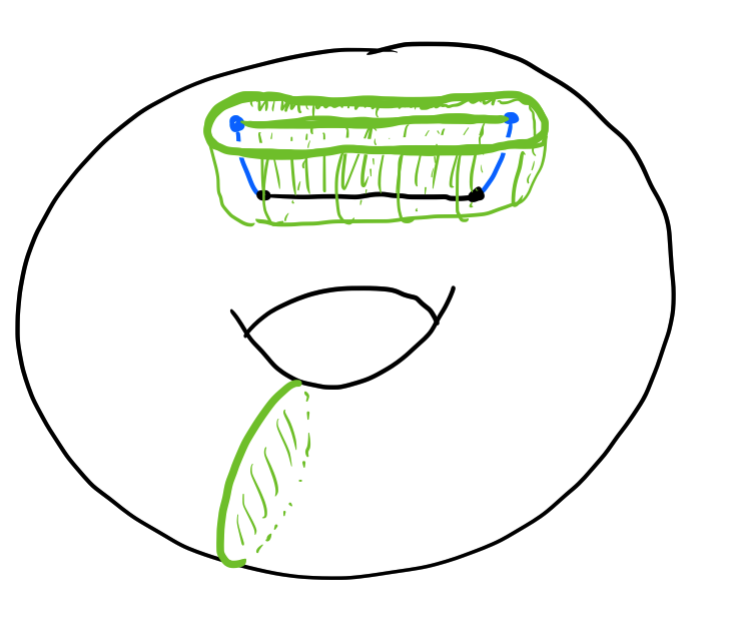}
    \caption{A depiction of the notation in Definition \ref{def:normal}.}
    \label{SFSNotation2}
\end{figure}

Since $\phi_1$ is $\boundary$-parallel in $U$, there is an isotopy rel $K_{13}$ of $\Delta$ so that it is normal with respect to $U$. Out of all such discs that are normal with respect to $U$, choose $\Delta$ to minimize
\[
c(\Delta) = (|\Delta \cap U|, |\Delta \cap Q|, |\Delta \cap A_2|).
\]
We compare tuples lexicographically.

\textbf{Claim 1:} Either $\Delta \cap X(U)$ is incompressible and $\boundary$-incompressible in $X(U)$ or $\Delta \cap S_U = \epsilon_U$.

Suppose, to the contrary, that $E \subset X(U)$ is a compressing disc for $\Delta \cap X(U)$. Since $\Delta$ is a disc, $\boundary E$ bounds a subdisc $E' \subset \Delta$. Since $E$ is a compressing disc, $E'$ intersects $U$. Thus an isotopy taking $E'$ to $E$, preserves normality but reduces $c(\Delta)$, a contradiction. Thus, $\Delta \cap X(U)$ is incompressible in $X(U)$.

Assume that $\Delta \cap S_U \neq \epsilon_U$. Suppose, to establish a contradiction, that $D \subset X(U)$ is a $\boundary$-compressing disc in $X(U)$ for $\Delta \cap X(U)$. $D$ is a disc whose boundary is the union of an arc $\delta = D \cap S_U$ and an essential arc in $\Delta \cap X(U)$. The simple closed curves $(\Delta \cap S_U) \setminus \epsilon_U$ divide $S_U$ into subsurfaces; let $S$ be subsurface containing $\epsilon_U$. The surface $S$ is one of the following:
\begin{enumerate}
    \item An annulus disjoint from $\epsilon_U$
    \item An annulus containing $\epsilon_U$
    \item A disc containing $\epsilon_U$
    \item A genus one surface with one boundary component, disjoint from $\epsilon_U$
\end{enumerate}

In fact, $S$ cannot be of type (1), since then $\Delta \cap X(U)$ would be compressible. The disc $D$ guides an isotopy of $\Delta$ rel $\boundary \Delta$ that effects the $\boundary$-compression. If $\delta$ joins distinct components of $(\Delta \cap S_U)\setminus \epsilon_U$, then after the isotopy, $\Delta$ is still normal with respect to $U$, but $|\Delta \cap U|$ has been reduced, a contradiction. If $\delta$ joins $\epsilon_U$ to a circle component of $\Delta \cap S_U$, then after the isotopy, $\Delta$ is still normal with respect to $U$, but the arc $\epsilon_U$ has been converted into an arc obtained by banding the original $\epsilon_U$ to the circle. The isotopy again reduces $|\Delta \cap U|$, a contradiction.  

Assume, therefore, that $\Delta$ joins a component of $\Delta \cap S_U$ to itself. In this case, a disc component of $\Delta \cap U$ is converted into an annulus $\alpha$. It can be checked that in each case, either one of the ends of $\alpha$ bounds a disc in $U \setminus \phi_1$, or the ends of $\alpha$ are parallel curves on $S_U$. 

Suppose that an end $c$ of $\alpha$ bounds a disc $E \subset U \setminus \phi_1$. Since $\Delta$ is a disc, $c$ also bounds a subdisc $E' \subset \Delta$. The disc $E'$ must have interior intersecting $S_U$. An isotopy of $E'$ to $E$ and then bumping slightly off $S_U$ creates a disc contradicting our choice of $\Delta$. On the other hand, if the ends of $\alpha$ are parallel simple closed curves in $S_U$, they cobound an annulus in $S_U$ disjoint from $\epsilon_U$. After performing an annulus swap on $\Delta$, we create a disc contradicting our initial choice of $\Delta$. \\
\qed(Claim 1)

\textbf{Claim 2:} $\psi$ is not $\boundary$-parallel in $X(U)$.

Suppose to the contrary that it is. Let $D$ be a disc of parallelism, chosen to intersect $Q$ minimally. Since $Q$ is incompressible, there are no circles of intersection. If there is an arc of intersection, it must span $Q$. An outermost such arc in $D$ then cuts off a $\boundary$-compressing disc for $Q$ in $X(\Gamma)$, a contradiction. \qed(Claim 2)

\textbf{Claim 3:} $\Delta \cap U$ contains at least one meridional disc

Suppose to the contrary that $\Delta \cap U$ consists only of the disc $\Delta_U$ and discs of Type (3) in Definition \ref{def:normal}. See Figure \ref{bandfig} for a depiction of the notation in this paragraph. Each such disc $D$ has boundary bounding a disc $D' \subset S_U$ containing $\epsilon_U$. Let $D$ and $D'$ be the outermost such discs. Then in $U$, $D \cup D'$ cobound a product region $B = D' \times [0,1]$, where $D' = D' \times \{0\}$. We may isotope the product structure, so that the disc $\Delta_U = \epsilon_U \times [0,1/2]$ and so that each other disc of $\Delta \cap U$, apart from $D \cup \Delta_U$, is the union of a vertical annulus and a horizontal disc. Extend the product structure to $D' \times [0, 3/2]$ using normal vectors to $D$. We may then isotope $\phi_2$, relative to its endpoints, to an arc $\phi'_2$ whose intersection with $D' \times [0,3/2]$ is equal to $\boundary \psi \times [0,3/2]$. This isotopy may change the graph $\Gamma$ by performing crossing changes between $e_1$ and $e_2$. The knot $\phi'_2 \cup \psi$ is isotopic to $K_{23}$ and is, therefore, an unknot. However, we see that it is also the result of banding the knot $K_{12} = \phi_1 \cup \psi$ to the knot $\phi'_2 \cup (\epsilon_U \times \{3/2\})$ using the band $\Phi = \epsilon_U \times [0,3/2]$. 

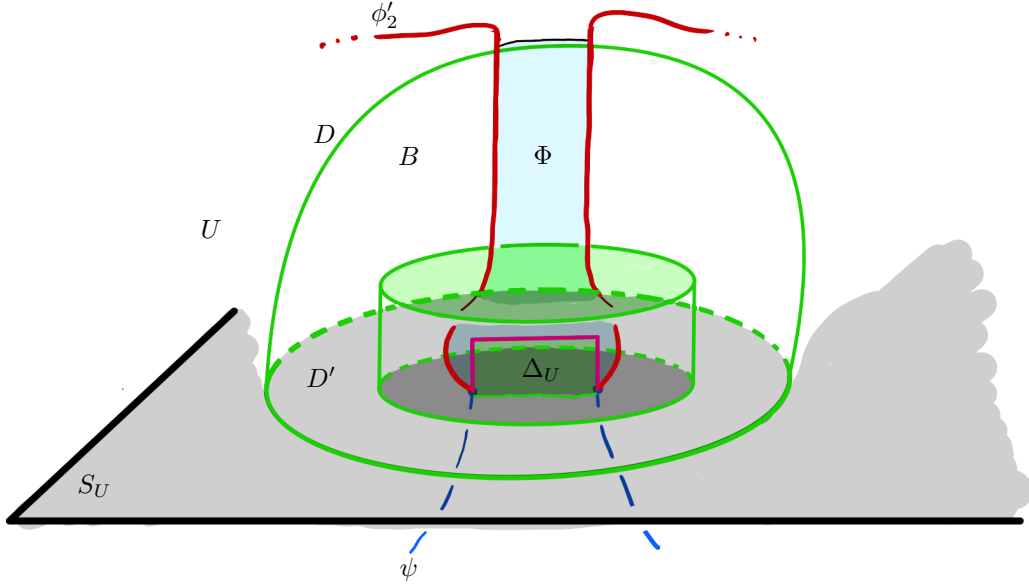
\begin{figure}[ht!]
\labellist
\small\hair 2pt
\pinlabel {$B$} at 310 346
\pinlabel {$\phi'_2$} [b] at 292 439
\pinlabel {$\Phi$} at 411 347
\pinlabel {$\Delta_U$} at 410 185
\pinlabel {$S_U$} at 73 97
\pinlabel {$U$} at 161 291
\pinlabel {$D$} [r] at 259 363
\pinlabel {$D'$} [b] at 244 168
%\pinlabel {$\phi_1$} [t] at 434 209
\pinlabel {$\psi$} [t] at 310 47
%\pinlabel {$V$} at 70 47
\endlabellist
    \includegraphics[scale=0.5]{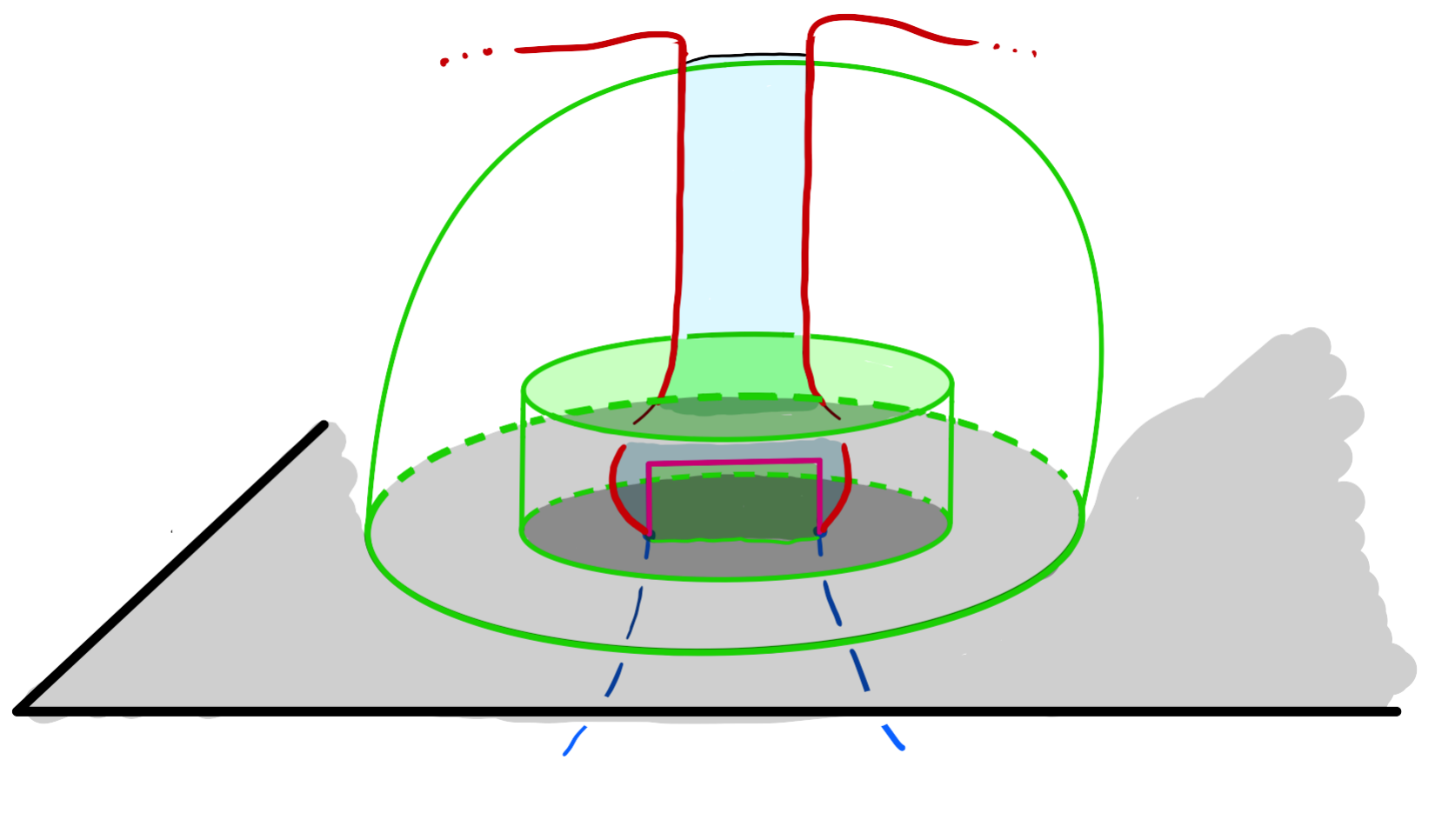}
\caption{ The knot $K(\phi'_2)$ can be obtained by banding the  knot $K(\phi_1)$ to a knot isotopic in $W$ to $\phi_1 \cup \phi'_2$. The band $\Phi$ is shown in blue. The green discs are $\Delta \cap W$. See the proof of Claim 3 in Case \ref{nonTUDA}.}
    \label{bandfig}
\end{figure}

Let $P$ be the sphere that is the frontier of a regular neighborhood of $\Delta$. It intersects the band $\Phi$ in $2(|(\Delta \cap U)\setminus \Delta_U| + 1$ arcs, two for each disc of Type (3) in $\Delta$ and one from the frontier of $\Delta_U$. Thus, the band $\Phi$ joins components of a split link and so $K_{23} = \phi'_2 \cup \psi$ is a band sum. A well-known theorem of Scharlemann \cite{Scharlemann-bandsum} says that the band sum is a connect sum. In particular, by the minimality of $\Delta$, $\Delta \cap S_U = \epsilon_U$. In particular, $\psi$ is $\boundary$-parallel in $X(U)$, contradicting Claim 2.
\qed (Claim 3)

\textbf{Claim 4:} $\Delta \cap A_1$ and $\Delta \cap Q$ consist of spanning arcs in $A_1$ and $Q$, respectively.

Since $\Delta \cap U$ contains a meridional disc and each meridional disc of $U$ intersects $A_1$, $|p| \geq 1$ times, $A_1$ contains at least $|p|$ spanning arcs. Thus, $Q$ contains at least one arc of intersection. The minimality of $c(\Delta)$ implies there are no circles of intersection between $\Delta$ and either of $Q$ or $A_1$. If there are inessential arcs of intersection in $Q$ or $A_1$, an isotopy would eliminate them and we would again contradict the minimality of $c(\Delta)$. \qed(Claim 4)

\textbf{Claim 5:} $\Delta \cap X(U\cup V)$ is incompressible in $X(U \cup V)$

This is the same as the first part of the proof of Claim 1.

\textbf{Claim 6:} $\Delta \cap X(U \cup V)$ consists of meridional discs of $X(U \cup V)$

The only incompressible orientable connected surfaces in the solid torus $X(U\cup V)$ are meridional discs, and $\boundary$-parallel discs and annuli. Since no component of $\Delta \cap A_2$ or $\Delta \cap Q$ is an inessential circle or inessential arc, no component of $\Delta \cap X(U \cup V)$ is a $\boundary$-parallel disc. Suppose some component $\Delta' \subset \Delta \cap X(U \cup V)$ is an outermost $\boundary$-parallel annulus in $X(U \cup V)$. By the minimality of $c(\Delta)$, it cannot be parallel to $Q$. Thus $\Delta'$ is parallel to an annulus in $\boundary X(U \cup V)$ that intersects $A_2$ and $Q$ in alternating rectangles. However, that implies that $\Delta \cap X(U)$ admits a $\boundary$-compressing disc, contradicting Claim 1. \qed (Claim 5)

We can now establish the contradiction, completing Case \ref{nonTUDA}. Let $\tau_m \co \boundary X(U \cup V) \to \boundary X(U\cup V)$ be the homeomorphism that is an $m$-fold twist along $Q$. Since $\Delta \cap A_2$ and $\Delta \cap Q$ both consist of spanning arcs, there exists $m \in \Z$ such that some component of $\tau_{-m}(\Delta \cap \boundary X(U \cup V))$ contains a curve that is a meridian of the solid torus $U \cup V$. Thus, since each component of $\boundary Q$ has slope $q/p$ with respect to $U \cup V$, each component of $\Delta \cap \boundary X(U \cup V)$ has slope $(mq + 1)/mp$. These simple closed curves must all be preferred longitudes of $U$ (i.e. meridians of $X(U \cup V)$), so we must have $mq + 1 = 0$. But this implies that $q = \pm 1$, a contradiction. \qed

\section{Sutured Manifold Theory}\label{sec: SMT}

In this section, independently from Section \ref{sec specific}, we use results from sutured manifold theory to prove:
\begin{proposition}\label{limited prop}
Suppose that $\Gamma \subset S^3$ is an Brunnian handcuff curve or $\theta$-curve. Suppose $X(\Gamma)$ contains an essential annulus, then one of the following occurs:
\begin{enumerate}
\item $\Gamma $ is a handcuff curve and a meridian of the separating edge is a guardrail for  an essential annulus $Q$.
\item $\Gamma $ is a toroidal $\theta$-curve
\item \label{limited possibilities} $\Gamma$ is an atoroidal $\theta$ curve and one of the following happens:
\begin{enumerate}
    \item $Q$ is of Type 1 with each component of $\boundary Q$ a meridian of $\Gamma$; or
    \item $Q$ is of Type 3-2 with guardrail a meridian of $\Gamma$
\end{enumerate}
\end{enumerate}
\end{proposition}
We then combine this result with those of Section \ref{sec specific} to complete the proof of Theorem \ref{Main theorem}. 

Sutured manifold theory, originally developed by David Gabai \cite{G1, G2, G3}, contains a set of powerful techniques that are useful for studying the euler characteristics of essential surfaces in Haken 3-manifolds. We do not need to revisit the technical machinery of the theory, but need only two results concerning the addition of a 2-handle to a suture. Here is what we need. The terminology is derived from \cites{G1, Scharlemann1, Scharlemann2, T1, T2}. 

Let $N$ be a compact orientable 3-manifold. In our setting, this 3-manifold will be the exterior of a genus two handlebody knot, two-component link, or knot in $S^3$. In such cases $H_2(N, \boundary N;\Z)$ is either $\Z^2$ (in the first two cases) or $\Z$ (in the last case). If $\boundary N$ is the union of two tori, we say that a class $\xi \in H_2(N, \boundary N)$ is \defn{Seifert-like} if the projection of $\boundary \xi$ to the first homology of each component of $\boundary N$ is nontrivial. 

The \defn{Thurston norm} $\chi_-(S)$ of a compact, connected, oriented surface $S$ is 0 if it is sphere or disc and $-\chi(S)$ otherwise. The Thurston norm of a disconnected compact, oriented surface is the sum of the Thurston norms of its components \cite{Thurston}. The surface $S$ is \defn{norm-minimizing} if it is properly embedded in $N$ and, out of all oriented properly embedded surfaces $T \subset N$ with $\boundary T = \boundary S$ and $[S, \boundary S] = [T, \boundary T]$ in the homology group\footnote{All homology groups have integer coefficients.} $H_2(N,\boundary S)$ we have $\chi_-(S) \leq \chi_-(T)$.  The surface $S$ is \defn{taut} (in $N$) if it is incompressible, Thurston norm minimizing, and not an inessential sphere or disc.

Suppose that $\gamma \subset N$ is a disjoint collection of oriented curves, that $A(\gamma)$ is a fixed regular neighborhood of $\gamma$ in $\boundary N$ and that $T(\gamma)$ is the union of some torus components of $\boundary N$ each of which is disjoint from $\gamma$. We orient $\boundary A(\gamma)$ using the orientation of $\gamma$. The curves $\gamma$ are required to have the property that the components of $\boundary N \setminus (A(\gamma) \cup T(\gamma))$ are partitioned into subsurfaces $R_-(\gamma)$ and $R_+(\gamma)$ such that $R_-(\gamma)$ is given a consistent inward pointing orientation and $R_+(\gamma)$ is given a consistent outward pointing orientation. Furthermore the orientations of $\boundary R_\pm(\gamma)$ induced by the orientations of $R_\pm(\gamma)$ match those induced by $\gamma$. We denote all of this data with the concise notation $(N,\gamma)$ and call $(N,\gamma)$ a \defn{sutured manifold}. The sutured manifold $(N,\gamma)$ is \defn{taut} if $N$ is irreducible and the surfaces $R_-(\gamma)$, $R_+(\gamma)$, and $T(\gamma)$ are all taut. The curves $\gamma$ and the tori $T(\gamma)$ are called the \defn{sutures} of $(N,\gamma)$. In this paper, $\gamma$ will be curves of one of the forms shown in Figure \ref{fig:std sutures}.

\begin{figure}[ht!]
\centering
\includegraphics[scale=0.5]{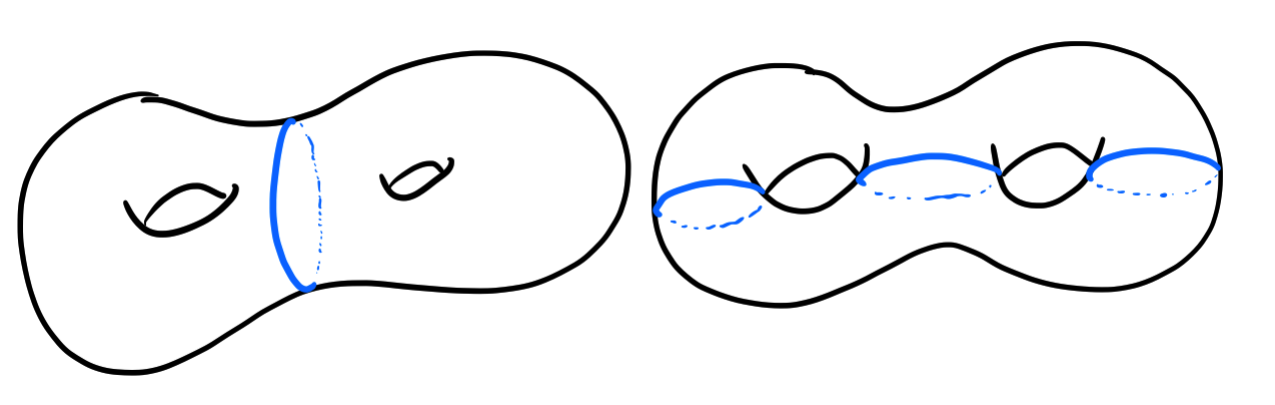}
\caption{The two options for standard sutures on a genus two boundary component}
\label{fig:std sutures}
\end{figure}

 We will be attaching a 2-handle to a suture. Observe that if $(N,\gamma)$ is a sutured manifold and if $b$ is a component of $\gamma$, then by choosing the 2-handle addition so that the attaching region of the 2-handle coincides with the neighborhood of $b$ in $A(\gamma)$, we have that $(N[b], \gamma \setminus b)$ is a sutured manifold.

Given a properly embedded oriented surface $S \subset N$, with $\boundary S$ transverse to $\gamma$, there is a natural choice of sutures $\gamma'$ for the sutured manifold $N\setminus S$ obtained by removing a regular open neighborhood of $S$ from $N$. The curves $\gamma'$ are isotopic to the double-curve sum of $\gamma$ with $\boundary S$. We say that $(N,\gamma) \stackrel{S}{\to} (N\setminus S, \gamma')$ is a \defn{sutured manifold decomposition} with \defn{decomposing surface} $S$. If $(N,\gamma)$ and $(N\setminus S, \gamma')$ are both taut, we say that the sutured manifold decomposition $(N,\gamma) \stackrel{S}{\to} (N\setminus S, \gamma')$ is \defn{taut}. Observe that if the decomposition is taut then so is $S$, though the converse may not hold. A \defn{sutured manifold hierarchy} is a sequence $\mc{H}$ of sutured manifold decompositions terminating in a sutured manifold $(N_n, \gamma_n)$ with $H_2(N_n,\boundary N_n) = 0$. Associated to $\mc{H}$ is a branched surface $B(\mc{H})$ obtained from the union of the decomposing surfaces in the hierarchy by smoothing according to orientations of the surfaces.

Taylor proved the following in \cite{T1}, the statement here follows that given in \cite[Theorem 2.1]{T2}. In the statement, we let $\beta \subset N[b]$ be the properly embedded arc that is the cocore of a 2-handle attached along the curve $b$.

\begin{theorem}\label{SMT}
Suppose that $(N,\gamma)$ is a taut sutured manifold with a genus 2 boundary component $F$. Suppose that $\gamma \cap F$ consists of either one separating or three non-separating curves and that $b$ is a component of $\gamma \cap F$. Let $Q \subset N$ be a properly embedded surface having no component a sphere or disc disjoint from $\gamma$. Assume also that $\boundary Q$ intersects $\gamma$ minimally up to isotopy and that $|\boundary Q \cap b| \geq 1$. Then one of the following occurs:
\begin{enumerate}
\item $Q$ has a compressing or $b$-boundary compressing disc
\item The pair $(N[b], \beta)$ has a proper summand $(M'_1, \beta'_1)$ where $M'_1$ is a lens space (other than $S^3$ or $S^1 \times S^2$) and $\beta'_1$ is the core of a solid torus in a genus one Heegaard splitting of $M'_1$.
\item The sutured manifold $(N[b], \gamma \setminus b)$ is taut. The arc $\beta$ can be properly isotoped to be disjoint from the first decomposing surface $\wihat{S}$ in some taut sutured manifold hierarchy $\mc{H}$ of $(N[b], \gamma \setminus b)$. If $|\gamma \cap F| = 3$, the surface $\wihat{S}$ can be taken to represent $\pm y$ for any given nonzero class in $H_2(N[b], \boundary N[b])$. If $|\gamma \cap F| = 1$, the surface $\wihat{S}$ can be taken to represent any Seifert-like homology class in $H_2(N[b], \boundary N[b])$.
\item $-2\chi(Q)+ |\boundary Q \cap \gamma| \geq 2|\boundary Q \cap b|$.
\end{enumerate}
\end{theorem}

Given a handcuff curve or $\theta$-curve $\Gamma \subset S^3$, we can create a sutured manifold as follows. Set $N = X(\Gamma)$. If $\Gamma$ is a handcuff curve, let $\gamma$ be a meridian of the separating edge of $\Gamma$. Orient $\gamma$, choose a regular neighborhood $A(\gamma)$, and orient the components (each a connected genus one surface with a single boundary component) of $\boundary N \setminus A(\gamma)$ to create $R_-(\gamma)$ and $R_+(\gamma)$. If $\Gamma$ is a $\theta$-curve, let $\gamma \subset \boundary N$ be the union of three simple closed curves, each a meridian of one of the edges of $\Gamma$. Choose a regular neighborhood $A(\gamma)$ of $\gamma$. Each component of $\boundary N \setminus A(\gamma)$ is then a pair of pants (a connected genus zero surface with three boundary components). Given one an inward orientation and call it $R_-(\gamma)$ and the other an outward orientation and call it $R_+(\gamma)$. These choices induce an orientation of $\gamma$ turning $(N,\gamma)$ into a sutured manifold. We call this the \defn{standard sutured manifold} corresponding to the graph $\Gamma$. See Figure \ref{fig:std sutures}.

\begin{lemma}\label{std taut}
Suppose that $\Gamma  \subset S^3$ is a Brunnian handcuff curve or $\theta$-curve, then the standard sutured manifold corresponding to $\Gamma$ is taut.
\end{lemma}

\begin{proof}
Let $(N,\gamma)$ be the standard sutured manifold corresponding to $\Gamma$. Observe that each of $R_-(\gamma)$ and $R_+(\gamma)$ is a connected surface of Euler characteristic equal to $-1$. If one of them, call it $S$, is not Thurston norm minimizing, then there would be homologous surface $T$ with the same boundary, having strictly small Thurston norm. Such a surface must be the union of spheres, discs, annuli, and tori.  Since $\boundary T$ is the meridians of edges of $\Gamma \subset S^3$, it must be the case that at least one component of $T$ is a disc. This implies $\Gamma$ is a reducible handcuff curve and, therefore, as observed previously not Brunnian. Thus, if $\Gamma$ is Brunnian then $R_-(\gamma)$ and $R_+(\gamma)$ are Thurston norm-minimizing. This also implies that they are incompressible. Since every tame sphere in $S^3$ bounds a ball to both sides, $N$ is irreducible. Thus, $(N,\gamma)$ is taut.
\end{proof}

Suppose that $(N,\gamma)$ is the standard sutured manifold corresponding to a Brunnian handcuff curve or $\theta$-curve $\Gamma \subset S^3$. Suppose $b \subset \gamma$ is a suture. Since $b$ is a meridian of an edge $e$ of $\Gamma$, attaching a 2-handle to $\boundary N$ along $b$ corresponds to removing the edge $e$ from $\Gamma$.  Thus, $N[b]$ is either the exterior of the unlink of two components (if $\Gamma$ is a Brunnian handcuff curve) or the exterior of the unknot (if $\Gamma$ is a Brunnian $\theta$-curve). 

We thus have the following adaptation of Theorem \ref{SMT}.

\begin{theorem}\label{graphthm}
Suppose that $\Gamma  \subset S^3$ is a Brunnian $\theta$-curve or handcuff curve. Let $(N,\gamma)$ be the standard sutured manifold corresponding to $\Gamma$. Let $b \subset \gamma$ be a suture and suppose that $Q \subset N$ is a properly embedded surface with $\boundary Q$ intersecting $\gamma$ minimally and with $|\boundary Q \cap b| \geq 1$. Suppose that $Q$ is incompressible and admits no $b$-boundary compressing disc. Then 
\[-2\chi(Q)+ |\boundary Q \cap \gamma| \geq 2|\boundary Q \cap b|.\]
\end{theorem}
\begin{proof}
By Lemma \ref{std taut}, $(N,\gamma)$ is taut. Let $e$ be the edge of $T$ that has $b$ as a meridian. The arc $\beta$ is equal to the edge $e$ minus a regular neighborhood of its endpoints.  We apply Theorem \ref{SMT}. Since $S^3$ has no lens space summands, we need only consider the possibility that:
\begin{itemize}
\item The sutured manifold $(N[b], \gamma \setminus b)$ is taut. The arc $\beta$ can be properly isotoped to be disjoint from the first decomposing surface $\wihat{S}$ in some taut sutured manifold hierarchy $\mc{H}$ of $(N[b], \gamma \setminus b)$. If $|\gamma \cap \boundary N| = 3$, the surface $\wihat{S}$ can be taken to represent $\pm y$ for any given nonzero class in $H_2(N[b], \boundary N[b])$. If $|\gamma \cap \boundary N| = 1$, the surface $\wihat{S}$ can be taken to represent any Seifert-like homology class in $H_2(N[b], \boundary N[b])$.
\end{itemize}

If $\Gamma $ is a handcuff curve, then $\Gamma \setminus e$ is a split link, and we see that $N[b]$ is reducible. Consequently, $(N[b], \nil)$ is not taut. Thus, we need only consider the situation when $\Gamma $ is a $\theta$-curve; i.e. $|\gamma \cap \boundary N| = 3$. We take $y \in H_2(N[b], \boundary N[b])$ to be the class represented by an oriented disc $\wihat{S} = S \cap N[b]$ where $S$ is a disc in $S^3$ whose boundary is the cycle $\gamma\setminus e$. Without loss of generality, we assume that $\wihat{S}$ has been isotoped to minimize $|\boundary S \cap \gamma|$. Since $\chi_-(\wihat{S}) = 0$, every norm minimizing surface in the class $[\wihat{S},\boundary \wihat{S}]$ has Thurston norm equal to 0. Since $N[b]$ is a solid torus, the first surface in the hierarchy $\mc{H}$ is a disc, isotopic relative to its boundary, to $\wihat{S}$, possibly after reversing orientation. But this implies that $N$ is $\boundary$-reducible, contradicting Lemma \ref{bd irred}.
\end{proof}

\begin{proof}[Proof of Proposition \ref{limited prop}]
Let $\Gamma$ be a Brunnian $\theta$-curve or handcuff curve such that $X(\Gamma )$ admits an essential annulus. Let $(N,\gamma)$ be a standard sutured manifold structure on $X(\Gamma )$. If $\Gamma $ is a handcuff curve, let $b = \gamma$. If $\Gamma $ is a $\theta$-curve, let $b_1, b_2, b_3$ be the components of $\gamma$. Let $Q$ be any essential annulus in $X(\Gamma )$ which, up to isotopy, intersects $\gamma$ minimally.

We start by showing that $Q$ has no $b$-$\boundary$-compressing disc, for any component $b$ of $\gamma$. Suppose that $D$ is such a disc. Boundary compress $Q$ using $D$, to obtain $Q'$, which is either a disc or the union of a disc and an annulus. Since, (by Lemma \ref{bd irred}) $X(\Gamma)$ is $\boundary$-irreducible and irreducible, the disc component is $\boundary$-parallel. As $Q$ is essential, this means that $Q'$ is the union of a disc and an annulus and, furthermore, $\boundary Q$ can be isotoped to intersect $\gamma$ in fewer points. This contradicts our choice of $Q$. 

By Theorem \ref{graphthm}, for any component $b \subset \gamma$
\begin{equation}\label{intersection ineq}
|\boundary Q \cap (\gamma \setminus b)| \geq |\boundary Q \cap b|.
\end{equation}

(Note that when $\boundary Q \cap b = \nil$, this statement is trivially true.) Consequently, if $\Gamma$ is a handcuff curve, $\boundary Q \cap b = \nil$. 

For the remainder, assume $\Gamma$ is a $\theta$-curve. By Corollary \ref{Guardrails exist}, we may assume that there exists a disc $A \subset \nbhd(T)$ that is a guardrail for $Q$; let $a = \boundary A$.

Each component $b_i$ of $\gamma$ bounds a disc $B_i$ in $\nbhd(T)$. We may assume that $B_i$ has been isotoped, relative to $\boundary B_i$, so that $B_i \cap A$ is a collection of arcs; by definition of $a$, this collection of arcs is minimal up to isotopy.  Any arc component of $(B_1 \cup B_2 \cup B_3) \cap A$ that is outermost in $A$ cuts of a subdisc of $A$ whose boundary is a properly embedded arc in $R_\pm$ with endpoints on the same component of $\gamma$. 

\textbf{Case 1:} $a \cap \gamma \neq \nil$. 

Without loss of generality, we may assume that $a \cap R_+$ contains an arc joining $b_1$ to itself. Each other arc of $a \cap R_+$ either joins $b_1$ to itself or joins $b_1$ to either $b_2$ or $b_3$. This implies $|a \cap b_1| > |a \cap (b_2 \cup b_3)|$. Hence, $a \cap R_-$ also contains an arc joining $b_1$ to itself.

Since $\boundary Q$ is disjoint from $a$, each arc of $\boundary Q \cap R_\pm$ (if any) joints $b_1$ to itself or joins $b_1$ to $b_2$ or $b_3$. If $\boundary Q \cap b_1 = \nil$, then each component of $\boundary Q \cap R_\pm$ is a circle parallel to $b_2$ or to $b_3$. This implies Conclusion (3a). 

% We could say something more about the relationship of a and $\boundary Q$.

Assume, therefore, that $\boundary Q \cap b_1 \neq \nil$. We will show that Conclusion (3b) holds. In this case, we use Theorem \ref{graphthm} to conclude:
\[
|\boundary Q \cap b_1| \geq |\boundary Q \cap (b_2 \cup b_3)|. 
\]
By Inequality \eqref{intersection ineq}, this would imply that $|\boundary Q \cap b_1| = |\boundary Q \cap (b_2 \cup b_3)|$. If such is the case, then $\boundary Q$ must intersect $b_j$ for some $j \in \{2,3\}$. Without loss of generality, assume that $j = 2$. Applying Inequality \eqref{intersection ineq} to $b_2$, we see that it must be the case that $|\boundary Q \cap b_2| = |\boundary Q \cap b_1|$. Consequently, $|\boundary Q \cap b_3| = 0$ and so $\boundary Q$ is disjoint from $b_3$. Thus, the arcs of $\boundary Q \cap R_\pm$ join $b_1$ to $b_2$ and are disjoint from $b_3$. Consequently, $\boundary Q$ intersects each of $b_1$ and $b_2$ always with the same sign and is disjoint from $b_3$. This is Conclusion (3b). 

\textbf{Case 2:} $a \cap \gamma = \nil$

Since each of $R_\pm$ is a pair-of-pants, this implies that $a$ is parallel to a component of $\gamma$. Without loss of generality, we may assume this component is $b_1$. Hence, either each component of $\boundary Q$ is isotopic to a component of $b_2 \cup b_3$, or $\boundary Q \cap R_\pm$ is a collection of arcs disjoint from $b_1$. For each choice of $\{j,k\} = \{2,3\}$, we may apply Inequality \ref{intersection ineq}, to see that $|\boundary Q \cap b_j| \geq |\boundary Q \cap b_k|$. Hence, $|\boundary Q \cap b_2| = |\boundary Q \cap b_3|$, and so each component of $\boundary Q \cap R_\pm$ is an arc joining $b_2$ to $b_3$. Hence, $\boundary Q$ intersects each of $b_2$ and $b_3$ always with the same sign and is disjoint from $b_1$. This is again Conclusion (3b).
\end{proof}

We can now assemble Proposition \ref{Step 1} and Proposition \ref{limited prop} to prove Theorem \ref{Main theorem}.

\begin{proof}[Proof of Theorem \ref{Main theorem}]
Suppose that $\Gamma$ is a Brunnian $\theta$-curve such that $X(\Gamma)$ is atoroidal. We will show that $X(\Gamma)$ is also anannular. Suppose, to the contrary, that $X(\Gamma)$ admits an essential annulus $Q$. By Lemma \ref{exists 123}, we may assume that $Q$ is of Type 1, 2, or 3. If it is of Type 3, then it is of type 3-2 or 3-3 by Lemma \ref{bd irred}. By Corollary \ref{Guardrails exist}, $Q$ admits a guardrail $A$. By Proposition \ref{limited prop}, $Q$ admits a guardrail which is a meridian of $\Gamma$. By Proposition \ref{Step 1} then implies that $X(\Gamma)$ contains an essential torus.
\end{proof}

\section{Uniqueness of Brunnian spines}\label{Spine Uniqueness}

How can we create multiple Brunnian spines for the same genus 2 handlebody knot? Here are two ways:
\begin{enumerate}
\item Start with a Brunnian handcuff curve and shrink the separating edge to obtain a Brunnian 2-bouquet.
\item Start with a Brunnian $\theta$-curve and shrink any one of the edges to obtain a Brunnian 2-bouquet.
\end{enumerate}

In \cite[Corollary 6.2]{T2}, Taylor derived the following as a consequence of Theorem \ref{SMT}:
\begin{theorem}\label{2h add}
Assume that $V$ is a genus 2 handlebody embedded in a compact, orientable 3-manifold $M$ such that $M$ contains no lens space summands, any pair of curves in $\boundary M$ that compress in $M$ are isotopic in $M$, and $N = X(V)$ is irreducible and $\boundary$-irreducible. If $a$ and $b$ are essential closed curves on $\boundary V$ bounding discs in $V$ that cannot be isotoped to be disjoint, then one of the following occurs:
\begin{enumerate}
\item One of $N[a]$ and $N[b]$ is irreducible and $\boundary$-irreducible.
\item There exists an essential annulus $A \subset N$ such that one component of $\boundary A$ lies on a component of $\boundary M$ and the other lies on $\boundary V$, is disjoint  from $a$ or $b$ and bounds a disc in $V$. 
\end{enumerate}
\end{theorem}

The following theorem follows easily.

\begin{theorem}
Suppose that $V \subset S^3$ is a genus two handlebody having a Brunnian spine. Then (up to isotopy in $V$) any two Brunnian spines of $V$ are either isotopic in $V$ or one is a 2-bouquet and is obtained from a Brunnian handcuff graph, also a spine of $V$, by shrinking the separating edge or from a Brunnian $\theta$-graph, also a spine of $V$, by shrinking an edge. In particular, a genus 2 Brunnian handlebody knot has at most one trivalent Brunnian spine and if it has two Brunnian 2-bouquet spines, then it has four Brunnian spines, three of which are 2-bouquets and one of which is a $\theta$-curve.
\end{theorem}
\begin{proof}
Let $\Gamma_1$ and $\Gamma_2$ be Brunnian spines of $W$ and set $V = \nbhd(\Gamma_1) = \nbhd(\Gamma_2)$ so that $N = X(\Gamma_1) = X(\Gamma_1)$. Without loss of generality, we may suppose that if one of them is $\theta$-curve or handcuff curve, then $\Gamma_2$ is. Let $a,b \subset \boundary V$ be meridians of edges of $\Gamma_1$ and $\Gamma_2$ respectively. Both of $N[a]$ and $N[b]$ are reducible or $\boundary$-reducible. By Lemma \ref{bd irred}, $N$ is irreducible and $\boundary$-irreducible. Thus, by Theorem \ref{2h add} with $M = S^3$, we see that $a$ and $b$ can be isotoped to be disjoint.

Let $\mu_i \subset X(\Gamma_i)$ consist of a single meridian for each edge of $\Gamma_i$. Since each curve of $\mu_1$ can be isotoped to be disjoint from each curve of $\mu_2$, we can isotope $\mu_1$ to be disjoint from $\mu_2$. Let $D_i$ be a set of meridian discs bounded by $\mu_i$. The isotopy can be extended into $V$, so that we have also isotoped $D_1$ to be disjoint from $D_2$. Move $\Gamma_1$ along with this isotopy. 

Suppose that $\Gamma_2$ is a $\theta$-curve or handcuff curve. Then $V \setminus \nbhd (D_2)$ is the union of two 3-balls $B$ and $B'$. The graph $\Gamma_2$ intersects each of $B$ and $B'$ in a tree with a single internal vertex and three edges. This graph is $\boundary$-parallel by the definition of spine. Since $\boundary V \setminus \nbhd(\boundary D_2)$ is the disjoint union of two pairs of pants, each curve of $\mu_1$ is isotopic to a curve of $\gamma_2$. Recall that the curves of $\mu_1$ are pairwise non-isotopic in $\boundary W$. Perform an isotopy of $V$ taking $\mu_1 \subset \mu_2$ and $D_1 \subset D_2$.  If $|\mu_1| = |\mu_2|$, then this isotopy may also be extended to an isotopy of $\Gamma_1 \cap B$ to $\Gamma_2 \cap B$ and $\Gamma_1 \cap B'$ to $\Gamma_1 \cap B$ by an isotopy fixing $D_1$ and $D_2$. Thus, if neither $\Gamma_1$ nor $\Gamma_2$ is a 2-bouquet the result holds. 
 
 Suppose that $\Gamma_2$ is again a $\theta$-curve or handcuff curve but $\Gamma_1$ is a 2-bouquet. By our previous argument, we have $D_1 \subset D_2$ with $|\mu_1| = 2$ and $|\mu_2| = 3$. Thus, $V \setminus \nbhd(D_1)$ is a 3-ball $B$, with two copies of the discs $D_1$ in its boundary. The graph $\Gamma_1 \cap B$ is a tree with four vertices and a single internal vertex that is $\boundary$-parallel. The graph $\Gamma_2 \cap B$ is a tree with six vertices two of which are internal. It is also $\boundary$-parallel. Thus, $\Gamma_1$ is isotopic to the result of shrinking the internal edge of $\Gamma_2 \cap B$ in $\Gamma_2$.
 
 Finally, suppose that both $\Gamma_1$ and $\Gamma_2$ are 2-bouquets. The argument is similar to what we have already done, except $V \setminus \nbhd(D_2)$ is a single 3-ball $B$ with two copies of $D_2$ in it boundary. The graph $\Gamma_2 \cap B$ is a tree with four edges and a single internal vertex which is $\boundary$-parallel in $B$. The curves $\mu_2$ lie in the surface $P = \boundary V \setminus \nbhd(\mu_2)$ which is a connected planar surface with four boundary components. If each curve of $\mu_1$ is isotopic to a curve of $\mu_2$, then, as before, we may isotope $D_1$ to $D_2$ and then $\Gamma_1$ to $\Gamma_2$ as before.
 
 Suppose, therefore, that $\mu_1$ is not isotopic to $\mu_2$. Let $\alpha$ be a component of $\mu_1$ not isotopic to a curve of $\mu_2$. Then $P \setminus \nbhd(\alpha)$ is the disjoint union of two pairs of pants. The components of $\mu_1$ are not isotopic to each other in $\boundary V$. Hence the curve $\alpha' = \mu_1 \setminus \alpha$ is isotopic to a curve of $\mu_2$. Perform the isotopy. Then $\mu_1$ and $\mu_2$ are each the union of two disjoint non-parallel non-separating curves in the genus 2 surface $\boundary V$. They share the component $\alpha'$. Thus, $\mu = \mu_1 \cup \mu_2$ cuts $\boundary V$ into the disjoint union of two pairs of pants, each of the three curves in $\mu$ is nonseparating, and each bounds a disc in $V$. Let $\Gamma$ be the $\theta$-curve in $V$ that is a spine for $V$ and is dual to those discs. The curves $\alpha$, $\alpha'$, and $\alpha'' = \mu_2 \setminus \alpha'$ are the meridians of the edges of $\Gamma$. Let $e$, $e'$, and $e''$ be the corresponding edges. Observe that $\Gamma_1$ can be obtained by shrinking $e''$ and $\Gamma_2$ can be obtained by shrinking $e$. The constituent knot of $\Gamma_1$ dual to $\alpha$ is obtained from $K' = e \cup e''$ by shrinking $e''$ to a point. Thus, in $S^3$, that knot and $K'$ are isotopic and so $K'$ is the unknot. Similarly, the constituent knot of $\Gamma_1$ dual to $\alpha'$ is obtained from $K = e' \cup e''$ by shrinking $e''$ to a point, so $K''$ is also the unknot. Likewise the constituent knots of $\Gamma_2$ are obtained from $K' = e \cup e''$ and $K'' = e \cup e'$ by shrinking $e$ to a point so they are both unknots. Consequently, the 3-constituent knots $K$, $K'$, and $K''$ of $T$ are all unknots so $\Gamma$ has the Brunnian property. Since $X(\Gamma) = X(\Gamma_1) = X(\Gamma_2)$ is $\boundary$-irreducible, $\Gamma$ is nontrivial, so it is Brunnian.
\end{proof}

\section*{Acknowledgements}
SAT thanks Alex Zupan and Ryan Blair for helpful conversations. This work was partially supported by NSF Grant DMS-2104022 and a Colby College Research Grant.

\end{document}